\documentclass[10pt,letterpaper]{amsart}

\usepackage{amssymb}
\usepackage{amsfonts}
\usepackage{amsmath}
\usepackage[all]{xy}
\usepackage[mathscr]{euscript}
\usepackage{stmaryrd}
\usepackage{mathtools}

\let\oldmarginpar\marginpar
\renewcommand\marginpar[1]{\-\oldmarginpar[\raggedleft\footnotesize #1]%
{\raggedright\footnotesize #1}}

\numberwithin{equation}{section}

   \newtheoremstyle{example}{\topsep}{\topsep}%
     {}
     {}
     {\bfseries}
     {}
     {\newline}
     {\thmname{#1}\thmnumber{ #2}\thmnote{ #3}}

\newtheorem{theorem}{Theorem}[section]
\newtheorem{proposition}[theorem]{Proposition}
\newtheorem{lemma}[theorem]{Lemma}
\newtheorem{corollary}[theorem]{Corollary}

\theoremstyle{definition}
\newtheorem{definition}[theorem]{Definition}

\theoremstyle{remark}
\newtheorem{rmk}[theorem]{Remark}

\theoremstyle{example}

\newcommand{\bfA}{\mathbf{A}}

\DeclareMathOperator{\struct}{\mathcal{O}}
\DeclareMathOperator{\GL}{\mathrm{GL}}
\DeclareMathOperator{\gl}{\mathfrak{gl}}
\DeclareMathOperator{\Spec}{\mathrm{Spec}}
\DeclareMathOperator{\Res}{\mathrm{Res}}
\DeclareMathOperator{\Tr}{\mathrm{Tr}}
\DeclareMathOperator{\ad}{\mathrm{ad}}
\DeclareMathOperator{\Ad}{\mathrm{Ad}}

\DeclareMathOperator{\res}{\mathrm{res}}

\DeclareMathOperator{\tMg}{\widehat{\mathcal{M}}}

\DeclareMathOperator{\A}{\mathcal{A}}
\DeclareMathOperator{\mmA}{\mathbf{\Gamma}}

\DeclareMathOperator{\M}{\mathcal{M}}
\DeclareMathOperator{\tM}{\widetilde{\mathcal{M}}}

\DeclareMathOperator{\dnu}{\d_e'}

\DeclareMathOperator{\diag}{\mathrm{diag}}
\DeclareMathOperator{\Hom}{\mathrm{Hom}}

\DeclareMathOperator{\pPo}{\pi_{\fP^1}}
\DeclareMathOperator{\ord}{\mathrm{ord}}

\DeclareMathOperator{\piU}{p_i^U}
\DeclareMathOperator{\piL}{p_i^L}
\DeclareMathOperator{\rk}{\mathrm{rank}}

\newcommand{\tN}{\widetilde{N}}

\newcommand{\TMi} {\T(\tM_i)}

\newcommand{\TGMg}{\T_{\gl_n}}

\renewcommand{\L}{\mathcal{L}}

\newcommand{\bPhi}{\bar{\Phi}}
\newcommand{\s}{\sigma}
\newcommand{\bp}{\bar{p}}
\newcommand{\vs}{\varsigma}
\newcommand{\tn}{\tilde{\nabla}}
\newcommand{\dy}{d_\De}
\newcommand{\Om}{\Omega}
\newcommand{\om}{\omega}
\newcommand{\Up}{\Upsilon}
\newcommand{\g}{\gamma}
\newcommand{\mA}{\Gamma}
\newcommand{\hmA}{\hat{\Gamma}}
\newcommand{\hr}{\hat{\rho}}
\renewcommand{\O}{\mathscr{O}}
\newcommand{\Oo}{\mathscr{O}^1}
\newcommand{\id}{\mathrm{I}}
\newcommand{\sgr}{\sg^{\bfr_\mathbf{e}}}

\renewcommand{\a}{\alpha}
\renewcommand{\b}{\beta}
\renewcommand{\d}{\delta}
\newcommand{\dAi}{A_{\tM, i}}
\newcommand{\dAj}{A_{\tM, j}}
\newcommand{\z}{\zeta}

\newcommand{\proj}{\mathbb{P}^1}
\newcommand{\n}{\nabla}
\newcommand{\pd}{\partial}

\newcommand{\C}{\mathbb{C}}
\newcommand{\Z}{\mathbb{Z}}
\renewcommand{\b}{\beta}
\newcommand{\hV}{\hat{V}}
\newcommand{\hg}{\hat{g}}
\newcommand{\hp}{\hat{p}}
\newcommand{\hn}{\hat{\n}}
\newcommand{\bn}{\bar{\n}}
\newcommand{\bnu}{\beta_\nu}
\newcommand{\bV}{\bar{V}}
\renewcommand{\o}{\varpi}
\renewcommand{\oe}{\varpi_E}
\newcommand{\ot}{\varpi_T}
\newcommand{\fo}{\mathfrak{o}}
\newcommand{\ft}{\mathfrak{t}}

\newcommand{\tfl}{\mathfrak{t}^\flat}
\newcommand{\Tfl}{T^\flat}
\newcommand{\fI}{\mathfrak{I}}
\newcommand{\fg}{\mathfrak{g}}
\newcommand{\fu}{\mathfrak{u}}
\newcommand{\fV}{\mathfrak{V}}
\newcommand{\hZ}{\widehat{Z}}

\newcommand{\pz}{\partial_z}

\newcommand{\Df}{D}

\newcommand{\De}{\Delta}

\newcommand{\fP}{\mathfrak{P}}
\newcommand{\fPD}{\fP_\De}
\newcommand{\PD}{P_\De}
\newcommand{\TD}{T_\De}
\newcommand{\bfPD}{\bar{\fP}_\De}
\newcommand{\ftD}{\ft_\De}

\newcommand{\dz}{d_z}
\newcommand{\bd}{\bar{d}}
\newcommand{\B}{\mathrm{B}}
\newcommand{\DR}{\mathrm{DR}}

\newcommand{\pfPo}{\pi_{\fP^1}}
\newcommand{\pfP}{\pi_{\fP}}
\newcommand{\bpit}{\bar{\pi}_\ft}
\newcommand{\bpiti}{\bar{\pi}_{\ft_i}}
\newcommand{\bpsi}{\bar{\psi}}
\newcommand{\tpsi}{\tilde{\psi}}
\newcommand{\pit}{\pi_\ft}
\newcommand{\piti}{\pi_{\ft_i}}
\newcommand{\bpt}{\bar{\mathrm{p}}_{\ft}}

\newcommand{\pt}{\mathrm{p}_{\ft}}

\newcommand{\tW}{\widetilde{W}}

\newcommand{\bfx}{\mathbf{x}}
\newcommand{\bfP}{\mathbf{P}}
\newcommand{\bfr}{\mathbf{r}}
\newcommand{\bfre}{\mathbf{r}_\bfe}
\newcommand{\bfg}{\mathbf{g}}
\DeclareMathOperator{\T}{T}
\newcommand{\dto}{d \hmA^0}
\newcommand{\JMU}{\mathrm{JMU}}

\newcommand{\bphi}{\bar{\phi}}
\newcommand{\bfe}{\mathbf{e}}
\newcommand{\bA}{\bar{A}}
\newcommand{\reg}{\mathrm{reg}}

\newcommand{\sg}{\gl_{n \proj}}

\newcommand{\tmu}{\tilde{\mu}}

\newcommand{\End}{\mathrm{End}}

\newcommand{\tA}{\tilde{A}}

\newcommand{\Th}{\Theta}
\newcommand{\ka}{\kappa}
\newcommand{\io}{\iota}
\newcommand{\I}{\mathcal{I}}
\newcommand{\bI}{\overline{\I}}
\newcommand{\J}{\mathcal{J}}
\newcommand{\tI}{\widetilde{\I}}
\newcommand{\tJ}{\widetilde{\J}}
\newcommand{\bN}{\overline{N}}
\newcommand{\tk}{\tilde{\ka}}
\newcommand{\tXi}{\tilde{\Xi}}

\newcommand{\vr}{\varrho}

\title[Isomonodromic deformations of connections]{Isomonodromic
  deformations of connections with singularities of parahoric formal
  type} \author{Christopher L.~Bremer}
\address{Department of Mathematics\\
  Louisiana State University\\
  Baton Rouge, LA 70803} \email{cbremer@math.lsu.edu} \author{Daniel
  S.~Sage} \email{sage@math.lsu.edu} \thanks{The research of the
  second author was partially supported by NSF grant~DMS-0606300 and
  NSA grant~H98230-09-1-0059.}  

\begin{document}
\begin{abstract}
  In previous work, the authors have developed a geometric theory of
  fundamental strata to study connections on the projective line with
  irregular singularities of parahoric formal type.  In this paper,
  the moduli space of connections that contain regular fundamental
  strata with fixed combinatorics at each singular point is
  constructed as a smooth Poisson reduction.  The authors then
  explicitly compute the isomonodromy equations as an integrable
  system.  This result generalizes work of Jimbo, Miwa, and Ueno to
  connections whose singularities have parahoric formal type.
\end{abstract}
\maketitle
\section{Introduction}
The study of transcendental solutions to differential equations has a
long pedigree in 
mathematics.  An important innovation in this field
from the turn of the century was Schlesinger's observation that
families of solutions to \emph{linear} differential equations often
satisfy interesting nonlinear equations.  His work on
monodromy-preserving deformations of Fuchsian differential equations
produced a remarkable family of nonlinear differential equations which
satisfy the Painlev\'e property, namely, that the only movable
singularities are simple poles \cite{Sc}.

For example, consider a family of regular singular differential equations on  $\proj$ 
of the form
\begin{equation}\label{schl}
\frac{d}{dz} \Phi = \sum_{i = 1}^{p} \frac{1}{z - x_i} A_i \Phi,
\end{equation}
where $x_1, \ldots, x_p \in \proj$ and $A_i$ is a matrix valued
holomorphic function in the coordinates $x_1, \ldots, x_p$.  A family
of fundamental solutions to \eqref{schl} has constant monodromy if and
only if the $A_i$'s satisfy the Schlesinger equations:
\begin{equation*}
dA_i = \sum_{j \ne i} [A_i, A_j] \frac{d (x_i - x_j)}{x_i - x_j}.
\end{equation*}
(See \cite[IV.1]{Sa} for a contemporary exposition.)  In general, we
will refer to the differential equations satisfied by a family of
linear differential equations with constant monodromy as isomonodromy
equations.

It took almost seventy years for progress to be made on the
isomonodromy problem for irregular singular differential equations.
In 1981, Jimbo, Miwa, and Ueno characterized the isomonodromy
equations for certain generic families of irregular singular
differential equations~\cite{JMU}.
One explanation for the long delay is that the monodromy map for
irregular singular point differential equations is significantly more
complicated; it involves the asymptotic behavior of solutions along
Stokes sectors at each singular point, so it is less explicitly
topological than the monodromy map in the regular singular case.  The
proof of Jimbo, Miwa and Ueno required a much clearer geometric
picture of the monodromy map.

In this paper, we will think of linear differential equations in terms
of meromorphic connections $\n$ on a trivial vector bundle $V$ on
$\proj$.  After fixing a basis for $V$, we may write $\n = d + \a$,
where $d$ is the usual exterior derivative and $\a$ is an
$\End(V)$-valued meromorphic one-form.  Roughly speaking, if one
considers the moduli stack $\M_\DR$ of meromorphic connections on
$\proj$, there is a formal Riemann-Hilbert map to the moduli stack
$\M_\B$ of irregular monodromy representations.\footnote{Here, $\DR$
  stands for the ``DeRham'' theory of meromorphic connections, and
  $\B$ stands for the ``Betti'' theory of irregular monodromy
  representations.}  An explicit description of the irregular
Riemann-Hilbert correspondence may be found in \cite{Mal}.

If one can find a smooth family $\M'$ that maps to $\M_{\DR}$, the
Malgrange-Sibuya theorem \cite{Mal} implies that the fibers of the
monodromy map are a foliation of $\M'$.  In particular, the
isomonodromy equations should correspond to an integrable distribution
on $\M'$.  By an observation of Boalch \cite[Appendix]{Boa}, $\L
\subset \M'$ is a leaf of this foliation if and only if the family of
connections corresponding to $\L$ is integrable, i.e., there exists a
connection $\bn$ on $\L \times \proj$ with the property that $\bn |_{
  \{x\} \times \proj}$ is a representative of the isomorphism class $x
\in \M_{\DR}$.  Throughout the paper, we will suppress the monodromy
point of view in favor of this integrability condition.  As an
example, we describe the smooth family of framed connections
$\tM_\JMU$ which was constructed by Boalch, building on work of Jimbo, Miwa and
Ueno.

Fix a finite collection of points $(x_i)_{i \in I} \subset \proj$ and
a vector of non-negative integers $(r_i)_{i \in I}$.\footnote{There is
  a slight simplification here: Jimbo, Miwa, and Ueno and Boalch allow
  $x_i$ to vary in $\proj$.}  Points in $\tM_\JMU$ correspond to
isomorphism classes of connections $\n = d + \a$, singular at $x_i$,
with the following additional data: there is a collection of framings
$g_i \in \GL_n(\C)$ such that the Laurent expansion of $\Ad(g_i) \a$
at $x_i$ has the form
\begin{equation*}
\Ad(g_i) \a = (A_{r_i} \frac{1}{(z- x_i)^{r_i}}  + \ldots + A_1 \frac{1}{z-x_i} +A_0) \frac{dz}{z-x_i},
\end{equation*}
where $A_i \in \gl_n(\C)$ and the leading term $A_{r_i}$ is a regular
diagonal matrix.  The moduli space $\tM_\JMU$ is, in fact, smooth. The
isomonodromy equations may be expressed in terms of a Pfaffian system
involving terms $\Th_i$, which control the dynamics of the framings,
and terms $\Xi_i$, which are essentially the principal parts of the
curvature of $\bn$~\cite{JMU}.

In \cite{BrSa}, the authors describe smooth moduli spaces of framed
connections with arbitrary slope, generalizing a construction of
\cite{Boa}.  The goal of this paper is to study isomonodromic
deformations of such connections.  The primary technical tool is a
local invariant of meromorphic connections called the fundamental
stratum, which plays the role of the leading term.  A stratum is a
triple $(P, r, \b)$, consisting of a parahoric subgroup $P \subset
\GL_n(\C[[z]])$, a non-negative integer $r$, and a `non-degenerate'
linear functional on the $r^{th}$ graded piece associated to the
canonical filtration on the Lie algebra of $P$.  The relevant
condition on connections which assures smoothness of the moduli space
is that $\n$ must contain a `regular' stratum.  This approach is
described in detail in Section~\ref{formal types}.

The primary motivation behind the introduction of fundamental strata
into the study of connections comes from the geometric Langlands
program, which in this case suggests an analogy between wildly
ramified adelic representations of $\GL_n$ and irregular monodromy
representations of rank $n$ on $\proj$ (see \cite{Fre} or, for a more
physical interpretation, \cite{Wi}).  The fundamental stratum
(alternatively, the minimal $K$-type) was originally used as a tool
for classifying wildly ramified representations of a reductive group
over a $p$-adic field \cite{Bu, MoPr}.  Thus, one hopes that a
dictionary between fundamental strata and families of monodromy
representations will better illuminate the wild ramification case of
the geometric Langlands correspondence.

We conclude with a short overview of the results in this paper.
Suppose that $\bfx = (x_i)_{i \in I} \subset \proj$ is a collection of
singular points, $\bfP = (P_i)_{i \in I}$ is a collection of uniform
parahoric subgroups $P_i \subset \GL_n(\C[[z - x_i]])$, and $\bfr =
(r_i)_{i \in I}$ is a vector of non-negative integers.  In
Section~\ref{three}, we give a construction of the moduli space
$\tM(\bfx, \bfP, \bfr)$ consisting of isomorphism classes of
connections with compatible framings on $\proj$ that contain regular
strata of the form $(P_i, r_i, \b_i)$ at $x_i$.  By
Proposition~\ref{mma} and Theorem~\ref{modspace}, $\tM(\bfx, \bfP,
\bfr)$ is a Poisson manifold; moreover, the symplectic leaves
correspond to connected components of
 moduli spaces of connections with a fixed formal
isomorphism class at each singular point.  When all of the parahoric
subgroups are maximal, i.e., $P_i = \GL_n(\C[[z-x_i]])$, $\tM(\bfx,
\bfP, \bfr) = \tM_\JMU$.

In Section~\ref{four}, we calculate the isomonodromy equations for
connections corresponding to points of $\tM(\bfx, \bfP, \bfr)$
(Theorem~\ref{defeq}).  In this context, the framing data is given by
a coset $U \backslash \GL_n(\C)$, where $U\subset \GL_n(\C)$ is the
unipotent subgroup determined by $P$.  In particular, part of the
isomonodromy data is a time dependent flow on an affine bundle over a
partial flag manifold.  Since $U$ is trivial when $P$ is a maximal
parahoric subgroup, this phenomenon is absent in the isomonodromy
equations of Jimbo, Miwa and
Ueno. 


In Section~\ref{integrability}, we describe a differential ideal $\bI$
on the moduli space $\tM (\bfx, \bfP, \bfr)$ corresponding to the
isomonodromy equations.
To be precise,  we construct a differential
ideal $\I$  on a principal $\GL_n(\C)$-bundle over $\tM(\bfx, \bfP, \bfr)$ and
then show that it descends.  The main result of the paper,
Theorem~\ref{int}, states that both $\bI$ and $\I$ are integrable Pfaffian
systems.  Moreover, there is a natural correspondence between
the leaves of the foliation determined by $\bI$  and
isomonodromic deformations of framed connections.
The proof is deferred to Section~\ref{descent}.

Finally, in Section~\ref{example}, we compute an explicit example of
the isomonodromy equations in the case where $P_i$ is an Iwahori
subgroup and $r_i = 1$ for all $i$.  To the authors' knowledge, this
is a completely new example of an integrable system on a Poisson
manifold.

\section{Formal Types}\label{formal types}

In this section, we review some results from the geometric theory of
fundamental strata and recall how they may be used to associate formal
types to irregular singular connections.  Let $F= k((z))$ be field of
formal Laurent series with coefficients in a field $k$ in
characteristic zero, and let $\fo \subset F$ be the corresponding
power series ring.  Let $\hV$ be an $n$-dimensional $F$ vector space.
A lattice chain $\mathscr{L} = \{L^i\}_{i \in \Z}$ is a collection of
$\fo$-lattices in $\hV$ satisfying the following properties: $L^i
\supset L^{i+1}$, and $z L^i = L^{i+e}$ with a fixed period $e > 0$.
We say $\mathscr{L}$ is {\em uniform} if $\dim_k L^i/L^{i+1} = n/e$
for all $i$; the lattice chain is {\em complete} if $e=n$.
\begin{definition}\label{parahoric}
  A uniform parahoric subgroup $P \subset \GL(\hV)$ is the stabilizer of a
  uniform lattice chain $\mathscr{L}$, i.e., $P=\{g \in \GL(\hV)\mid g
  L^i = L^i \text{ for all } i\}$. The Lie algebra of $P$ is the
  parahoric subalgebra $\mathfrak{P} \subset \gl(\hV)$ consisting of
  $\mathfrak{P} = \{p \in \gl(\hV) \mid p L^i \subset L^i \text{ for all
  } i\}$.  Note that $\mathfrak{P}$ is in fact an associative
  subalgebra of $\gl(\hV)$.  An Iwahori subgroup $I$ is the stabilizer
  of a complete lattice chain, and an Iwahori subalgebra
  $\mathfrak{I}$ is the Lie algebra of $I$.
\end{definition}

For more details on lattice chains and the corresponding parahoric subgroups and
subalgebras, see for example~\cite{Sa00,BrSa}.

There are natural filtrations on $P$ (resp. $\fP$) by
congruence subgroups (resp. ideals).  For $r\in \Z$, define the
$\mathfrak{P}$-module $\mathfrak{P}^r$ to consist of $X \in
\mathfrak{P}$ such that $X L^i \subset L^{i+r}$ for all $i$; it is an
ideal of $\fP$ for $r\ge 0$ and a fractional ideal otherwise.
The congruence subgroup $P^r \subset P$ is then defined by $P^0 = P$
and $P^r = \id_n + \fP^r$ for $r > 0$.
Note that these ideals are multiplicative, in the sense that $\fP^r \fP^s = \fP^{r+s}$.
If we fix a form $\nu\in \Omega^1_{F/k}$ of order $-1$,
the  pairing 
\begin{equation}\label{pair}
\langle X, Y \rangle_\nu = \Res \Tr(X Y \nu)
\end{equation}
identifies $\fP^{-r}$ with $(\fP^{r+1})^\perp$ and $(\fP^{r+1} /
\fP^{s+1})^\vee$ with $\fP^{-s} / \fP^{-r}$.  There are similar
formulas for $\nu$ of arbitrary order, for example,
$(\fP^{r+1})^\perp\cong\fP^{-r+(1+\ord(\nu))e}$.  Throughout this
section, we will assume for simplicity that $\ord(\nu)=-1$, but all
definitions and results can be stated for other $\nu$.

\begin{definition}\label{stratdef}
Let $\hV$ be an $F$ vector space, and let $(P, r, \b)$ be a triple consisting of
\begin{itemize}
\item $P \subset \GL(\hV)$  a uniform parahoric subgroup;
\item $r \in \Z_{\ge 0}$, with $\gcd (r, e) = 1$;
\item $\b \in (\fP^r/\fP^{r+1})^\vee$.
\end{itemize}
After fixing $\nu$ as above, we may identify $\b$ with a coset
$\bnu + \fP^{-r+1} \subset \fP^{-r}/\fP^{-r+1}$.  Thus, we may
choose a representative $\bnu \in \fP^{-r}$ for $\b$.
We say that $(P, r, \b)$ is a \emph{uniform stratum} if 
$\bnu + \fP^{-r+1}$ contains no nilpotent
elements.
\end{definition}
In this paper, we are interested in strata that satisfy a `graded' version
of regular semisimplicity.  Fix a totally ramified field extension $E/F$ of degree $e$,
and let $\fo_E$ be the corresponding integral domain.
Let 
\begin{equation}\label{torus}
T \cong (E^\times)^{n/e} \subset \GL(\hV)
\end{equation} be a maximal torus and 
let $\ft \cong E^{n/e} \subset \gl(\hV)$ be the corresponding Cartan subalgebra.  
We denote the identity elements of the Wedderburn components of $\ft$ by $\chi_j$, and we write
$T(\fo)$ (resp. $\ft(\fo)$) for $(\fo_E^\times)^{n/e} \subset T$ (resp. $\fo_E^{n/e} \subset \ft$).
Moreover, write $\tfl$ for the $k$-span of the $\chi_j$ and $\Tfl$ for
$(\tfl)^\times$.

Now, suppose that $(P,r, \b)$ is a uniform stratum.  There is a map
$\pd_{\b, s} : \fP^{s}/\fP^{s+1} \to \fP^{s-r}/ \fP^{s-r+1}$ given by
$\pd_{\b, s}(X+\fP^s) = \ad(X)(\bnu) + \fP^{s-r+1}$.
\begin{definition}
We say that $(P, r, \b)$ is a regular stratum centralized by $T$ if
it satisfies the following conditions:
\begin{enumerate}
\item $T(\fo) \subset P$;
\item $\ker (\pd_{\b, s}) = \ft \cap\fP^{s} + \fP^{s+1}$ for all $s$;
\item  $y_\b = z^r\bnu^e+ \fP^1$ is a semi-simple element of  the algebra $\fP/ \fP^1$;
\item when $r = 0$ (and hence, $e_P = 1$), the eigenvalues of $y_\b$ 
are distinct modulo $\Z$.
\end{enumerate}
\end{definition}

In fact, by \cite[Theorem 3.6]{BrSa}, $\mathscr{L}$ induces a complete
lattice chain on each $\chi_j (\hV)$ (with period $e$).  Write $I_j$
for the corresponding Iwahori subgroup.  It follows from \cite[Lemma
2.4]{BrSa} that it is the unique lattice chain $\mathscr{L}_j$ with
$\fo_E^\times \subset I_j$.  Moreover, if we fix a uniformizer $\oe$
for $\fo_E$, then $\oe \fI_j = \fI_j^1$.  By \cite[Proposition
1.18]{Bu}, we deduce that the matrix $\ot = (\oe, \ldots, \oe) \in
\ft$ satisfies the property $\ot \fP = \fP \ot = \fP^1$.

Recall the following proposition:
\begin{proposition}[{\cite[Proposition 2.10]{BrSa}}]
\label{cores}
Let $(P, r, \b)$ be a regular stratum centralized by $T$, and let
$\bnu \in \fP^{-r}$ be a representative for $\b$.
There is a morphism of $\mathfrak{t}$-modules  $\pi_{\mathfrak{t}} : \gl(\hV) \to \mathfrak{t}$ 
satisfying the following properties:
\begin{enumerate}
\item \label{cores1} $\pi_{\mathfrak{t}}$ restricts to the identity on $\mathfrak{t}$;
\item \label{cores2} $\pi_{\mathfrak{t}} (\mathfrak{P}^\ell) \subset \mathfrak{P}^\ell$;
\item \label{cores3} the kernel of the induced map
\begin{equation*}
\bpit : (\mathfrak{t} + \mathfrak{P}^{\ell-r})/ \mathfrak{P}^{\ell-r+1} \to
\mathfrak{t}/ (\mathfrak{t} \cap \mathfrak{P}^{\ell-r+1})
\end{equation*}
is given by the image of $\ad(\mathfrak{P}^\ell) (\bnu)$ modulo $\mathfrak{P}^{\ell-r+1}$;
\item \label{cores4} if $y \in \mathfrak{t}$ and $X \in \gl(\hV)$, then
$\langle y, X\rangle_\nu = \langle y, \pi_{\mathfrak{t}} (X)\rangle_\nu$;
\item \label{cores5} let $\bar{\pi}_{\ft,\ell}:\fP^\ell/\fP^{\ell+1}\to
  \ft/\ft\cap\fP^{\ell+1}$ be the induced map, and set
  $W_\ell=\ker(\bar{\pi}_{\ft,\ell})$.  Then, the induced map
  $\ad(\bnu):W_{\ell}\to W_{\ell-r}$ is an isomorphism.
\end{enumerate}
\end{proposition}

A connection on $\hV$ is a $k$-derivation $\n : \hV \to \hV \otimes
\Omega^1_{F/k}$.  If $\tau$ is a $k$-derivation on $F$, we write
$\n_\tau$ for the composition of $\n$ with the inner derivation
associated to $\tau: $ $\n_\tau (v) = i_\tau (\n (v))$.  In
particular, let $\tau_\nu$ be the derivation with the property that
$i_{\tau_\nu} (\nu) = 1$.
\begin{definition}
Let $(P, r, \b)$ be a uniform stratum.  When $r \ge 1$, 
we say that $(\hV, \n)$ contains  $(P, r, \b)$ if $\n_{\tau_\nu} (L^i) \subset L^{i-r} $
and $(\n_{\tau_\nu} - \b_\nu) (L^i) \subset L^{i-r+1}$ for all $i$.  When
$r = 0$, and thus $e = 1$, we say that $(\hV, \n)$ contains $(P, 0, \b)$
if $(\n_{\tau_\nu} - \b_\nu) (L^i) \subset L^{i+1}$  for some lattice $L^i$.
\end{definition}

Given a trivialization $\phi : \hV \overset{\sim}{\to} F^n$, we write $[\n]_\phi$
for the matrix of $\n$ with respect to the standard basis of $F^n$.   By the Leibniz rule,
$\n = \dz + [\n]_\phi$ where $\dz$ is the usual exterior $k$-differential on $F$.
The group $\GL_n(F)$ acts transitively on the space of trivializations for $F$, and
\begin{equation}\label{gaugetrans}
[\n]_{g \phi} = g \cdot [\n]_\phi := \Ad(g) [\n]_\phi - (\dz g) g^{-1}.
\end{equation}
If we have fixed a base trivialization, we will shorten $[\n]_{g\phi}$  to $[\n]_g$
and $[\n]_\phi$ to $[\n]$.
In general, the left action $g \cdot$ on $\gl_n(F) \otimes \Omega^1_{F/k}$
is called a gauge transformation, and we say that two matrices
$X, Y \in \gl_n(F) \otimes \Omega^1_{F/k}$  are gauge equivalent if there
exists $g \in \GL_n(F)$ such that $g \cdot X = Y$.  Thus,
if $\n$ and $\n'$ are connections on $\hV$, and $[\n]_\phi$ is gauge equivalent
to $[\n']_\phi$, then $(\hV, \n)$ and $(\hV, \n')$ are isomorphic.

Now, suppose that $(P, r, \b)$ is a regular stratum in $\GL_n(F)$
centralized by a torus $T$.  We denote the pullback of $P$ and $\b$ to
$\GL(\hV)$ by $P^\phi$ and $\b^\phi$, respectively.
\begin{theorem}{\cite[Theorem 4.13]{BrSa}}\label{exft}
  If $(\hV, \n)$ contains the stratum $(P^\phi, r, \b^\phi)$, then
  there exists $p \in P^1$ and a regular element $A_\nu \in \ft\cap\fP^{-r}$ such
  that $p \cdot [\n]_\phi = A_\nu \nu$.  Furthermore, the orbit of
  $A_\nu \nu$ under $P^1$-gauge transformations contains
  $(A_\nu+\fP^1) \nu $.
 \end{theorem}

 \begin{rmk}\label{intertwine} By \cite[Lemma 4.4]{BrSa}, the map $\fP^{-r}\to
   \fP^{-r}/\fP^{1}$ intertwines the gauge and adjoint actions of $P$.
   This implies that the functional induced on $\fP^r/\fP^{r+1}$ by
   $A_\nu$ coincides with $\b$.
\end{rmk}
We now  define the ``formal type'' of a connection.
 \begin{definition}
 Let $A \in \fP^\vee \backslash \{0\},$ and suppose that $\fP^{r+1}$
 is the smallest congruence ideal contained in $A^\perp$. Let $\b$ be
 the restriction of $A$ to $\fP^r/\fP^{r+1}$.  We say that 
 $A$ is a \emph{formal type} if it satisfies the following conditions:
 \begin{enumerate}
 \item the stratum $(P, r, \b)$ is regular and centralized by $T$, and
 \item any representative $A_\nu \in \fP^{-r}$ for $A$ lies in $\ft+ \fP^1$.
 \end{enumerate}
 The connection $(\hV, \n)$ has formal type $A$ if there is a
 trivialization $\phi : \hV \to F^n$ such that $(\hV, \n)$ contains
 the stratum $(P^\phi, r, \b^\phi)$, and $[\n]_\phi$ is formally gauge
 equivalent to an element of $(A_\nu + \fP^1) \nu$ by an element of
 $P^1$.  Finally, we say that two connections that contain regular
 strata are \emph{combinatorially equivalent} if the strata have the same
 parahoric subgroup and slope.
 \end{definition}
 By \cite[Corollary 4.16]{BrSa}, any two connections with formal type
 $A$ are formally isomorphic, and the formal type is independent of
 $\nu$.  The converse is false, since any conjugate of $A$ by the
 relative Weyl group of $T$ is also a formal type for $(\hV, \n)$.
 However, fixing a stratum uniquely determines the formal type.
For the rest of the section, we assume (without loss of generality) that $P \subset \GL_n(\fo)$.
\begin{proposition}\label{uniqft}
  If $\n$ contains a regular stratum $(P^\phi, r, \b^\phi)$, there
  exists a unique formal type $A$ such that $A |_{\fP^r} = \b$ and
  $[\n]_\phi$ is gauge equivalent to an element of $(A_\nu +
  \fP^{1})\nu$ by an element of $\GL_n(\fo)$.
\end{proposition}
\begin{proof}
By Theorem~\ref{exft} and Remark~\ref{intertwine}, $(\hV, \n)$ has a
formal type $A$ such that $A|_{\fP^r}=\b$, and there exists $p \in P^1$ 
such that  $p \cdot [\n]_\phi = A_\nu \nu$, where $A_\nu\in \ft\cap\fP^{-r}$.
This formal type is unique, since \cite[Lemma 3.16]{BrSa} implies that
if $[\n]_\phi$ is gauge equivalent to $A'_\nu \in \ft\cap\fP^{-r}$ and
$A'_\nu |_{\fP^r} = \b$, then $A'_\nu + \fP^1 = A_\nu + \fP^1$.
\end{proof}

Finally, throughout the paper we will need to consider a slight variation
on the formal type.  Choose a uniformizer $\oe$ for $E$ such that
$\oe^e = z$; under a suitable embedding
$E\hookrightarrow\gl_{n/e}(\C)$, 
such a way that
\begin{equation}\label{varpip}
\oe= 
\begin{pmatrix}
0 & 1 & \cdots & 0\\
\vdots & \ddots & \ddots & \vdots \\
0 & \ddots & 0 & 1  \\
z  & 0 & \cdots& 0 
\end{pmatrix}.
\end{equation}
We choose a basis for each $\chi_j\hV$ (and hence for $V$) such that
$\ot$ is block diagonal with these blocks.  Using this basis, we
define $H_T=(H_E,\dots,H_E)\in\fP$ as the block diagonal matrix with
blocks given by the diagonal matrix $H_E=\diag(\frac{e-1}{2e},
\frac{e-3}{2 e}, \ldots, \frac{1-e}{2 e})$. Let $H_T' \in \fP^\vee$ to
be the corresponding functional $H_T'(X) = \Tr (H_T X)|_{z = 0}$.
\begin{definition}
Let $(\hV, \n)$ have formal type $A$.  We define the \emph{normalized formal type}
of $(\hV, \n)$ to be  $\tA = A + H_T'$.
\end{definition}
Note that if $e=1$, then $\tA=A$.
\begin{proposition}\label{gA}
Suppose that $(\hV, \n)$ contains a regular stratum and has normalized
formal type $\tA$.  If $\tA_\nu \in \fP^{-r}$
is a representative for $\tA$, then there exists $\hp \in P^1$
such that $\hp \cdot [\n]_\phi = \tA_\nu \nu$.  
\end{proposition}
\begin{proof}
  If $r=0$ (so $e=1$), this is just Theorem~\ref{exft}, so assume that
  $r\ge 1$.  Note that $H_T' \in (\tfl)^\perp$.  Therefore, part 4 of
  Proposition \ref{cores} imply that $\pi_\ft (H_T) \in \fP^1$.  Part
  3 of the same proposition shows that there exists $X \in \fP^r$ such
  that $\ad (X) \tA_\nu \in H_T \frac{1}{\Res(\nu)} + \fP^1$, so there
  exists $p \in P^1$ such that $p \cdot [\n]_\phi = A_\nu \nu$ by
  Theorem~\ref{exft}.  Moreover, $\exp(-X) \cdot \tA_\nu \nu \in
  (A_\nu +\fP^1)\nu$, so there exists $p' \in P^1$ such that $p'
  \exp(-X) \cdot \tA_\nu \nu= A_\nu \nu$.  It follows that $(\exp(X)
  (p')^{-1} p) \cdot [\n]_\phi = \tA_\nu \nu$.
\end{proof}

\begin{proposition}\label{stab}
Any normalized formal type $\tA$ is stabilized by $T \cap P^1$.
\end{proposition}
\begin{proof}
Since $\ad (\fP^1) H_T \in \fP^1$, $P^1$ stabilizes $H_T + \fP^1$.
Moreover, the corresponding formal type $A$ is stabilized
by $T$.
\end{proof}
\section{Moduli spaces of connections}\label{three}
In this section, we will describe the moduli space of connections on
$\proj(\C)$ with compatible framings and fixed combinatorics at each
singular point.

First, we recall the construction of the moduli space of framed
connections on $\proj(\C)$ with fixed formal type at each singular
point~\cite{BrSa}.  Throughout, $I$ will be a finite indexing set and $\bfx =
\{x_i\}_{i \in I}$ a collection of distinct points in $\proj$.  We
denote the completion of the function field of $\proj$ at $x_i$ by
$F_i$ and the corresponding power series ring by $\fo_i$.
 
 Let $V$ be a trivial rank $n$ vector bundle on $\proj$, i.e., we have
 fixed a trivialization $V \cong \struct_{\proj}^n$.  Accordingly, we may
 identify the space of all global trivializations of $V$ with
 $\GL_n(\C)$. Note that a global trivialization determines a local
 trivialization of $V_{i} := V \otimes_{\struct_{\proj}} F_{i}$, so
 there is a natural inclusion $\GL_n(\C) \hookrightarrow \GL_n(F_i)$;
 moreover, the global sections of $V$ generate a distinguished lattice
 $L_i \subset V_i$.  Suppose that $\n$ is a connection on $V$ with the
 property that the induced connection on $V_i$ has formal type $A_i$.
 We will assume without loss of generality that the associated torus $T_i$ is contained in
 $\GL_n(\fo_i)$.  Therefore, by \cite[Proposition 4.14]{BrSa}, $A_i$ determines
 a unique stratum $(P_i, r, \b_i)$ in $\GL_n(F_i)$. We also set $\o_i=\o_{T_i}$.
 
 Throughout this section, $U_i$ will denote the unipotent subgroup
 $P^1_i \cap \GL_n(\C)$ with Lie algebra $\fu_i=\fP_i^1 \cap
 \gl_n(\C)$.  For simplicity, we write $\fg_i$ for the parahoric
 subalgebra $\gl_n(\fo_i)$; its radical is $\fg_i^1=t\fg_i$.  Note
 that $U_i \cong P^1_i / (1 + \fg_i^1)$ and $\fu_i \cong \fP^1_i /
 \fg^1$.
 
 \begin{definition}\label{framingdef}
A \emph{compatible framing} for $\nabla$ at $x_i$ is a global trivialization  $g \in \GL_n (\C)$
with the property that $\n$ contains the $\GL(V_i)$-stratum $(P^g_i, r, \beta^g_i )$
defined above.
We say that $\n$ is \emph{framed} at $x_i$ if there exists such a $g$.
\end{definition}
 
We now define the moduli space of connections with fixed formal type
and a specified framing at each singular point.  Set $\bA_i = A_i
|_{\fP^{1}}$.

 \begin{definition}\label{tmdef}
 Let $\tM(\bfA)$ be the the moduli space of isomorphism classes
 of triples $(V, \n, \bfg)$, where
 \begin{itemize}
 \item $\n$ is a meromorphic connection on the trivial bundle $V$ 
 with singularities at $\{x_i\}_{i \in I}$;
 \item $\bfg = \{U_i g_i\}_{i \in  I},$ with $g_i$ a compatible framing for $\n$ at $x_i$;
 \item the formal type $A'_i$ of $\n$ at $x_i$ satisfies $\bA'_i = \bA_i$.
 \end{itemize}
 
 \end{definition}
 The moduli space $\tM(\bfA)$ is built out of ``extended'' coadjoint
 orbits $\tM(A_i)$ determined locally by each formal type $A_i$.  In
 the following, let $A$ be a formal type.  We define $\pfP : \fg^\vee \to \fP^\vee$ and $\pfPo :
 \fg^\vee \to (\fP^1)^\vee$ to be the restriction maps and $\Oo$ to
 be the orbit of $\bA$ under the coadjoint action of $P^1$.
 Also, when $r = 0$ (so $e= 1$), we take $(\tfl)' \subset
 \gl_n(\C)^\vee$ to be the set of functionals of the form $\phi(X) =
 \Tr (DX)$, where $D \in \tfl$ is a diagonal matrix with distinct
 eigenvalues modulo $\Z$.
  \begin{definition}\label{tM}
  When $r > 0$, 
  define the extended orbit
  $\tM (A) \subset (U \backslash
  \GL_n(\C)) \times \mathfrak{g}^\vee$ to be
\begin{equation*}
\tM(A) = \{ (U g, \alpha) \mid \pi_{\mathfrak{P}^1} (\Ad^*(g) (\alpha)) \in \Oo) \}.
\end{equation*}
When $ r = 0$, we define $\tM(A)$ by
\begin{equation*}
\tM(A)  = \{ (g, \a) \in \GL_n(\C) \times \gl_n(\C)^\vee \mid 
\Ad^*(g) \a \in (\tfl)' \} .
\end{equation*}

There is a natural action $\rho$ of $\GL_n(\C)$ on $\tM(A)$ given by
$\rho(h) (U g, \a) = (U g h^{-1}, \Ad^*(h) \a)$.
  \end{definition}
  Note that when $r = 0$, $\tM(A)$ is independent of $A$.
  \begin{proposition}\label{mm}{\cite[Proposition 5.10]{BrSa}}
    The space $\tM(A)$ is a symplectic manifold, and $\rho$ is a
    Hamiltonian action.  The moment map for $\rho$ is given by
    $\mu_\rho (U g, \a) = \res(\a)$, where
    $\res(\a)=\a|_{\gl_n(\C)}$.
  \end{proposition}

\begin{rmk} If $\a_\nu$ is any representative of $\a$, then $\res(\a)=\Res(\a_\nu\nu)$.
\end{rmk}

 \begin{theorem}\label{modthm}{\cite[Theorem 5.4]{BrSa}}
 The moduli space $\tM(\bfA)$ is  the symplectic reduction of $\prod_{i\in I} \tM(A_i)$
by the diagonal action of $\GL_n(\C)$:
\begin{equation*}
\tM (\mathbf{A}) \cong (\prod_{i \in I} \tM(A_i)) \sslash_0 \GL_n(\C).
\end{equation*}
Moreover, $\tM(\bfA)$ is a symplectic manifold.
\end{theorem}
Specifically, $\prod_{i \in I} \tM(A_i)$ is a symplectic manifold, and the diagonal
action of $\GL_n(\C)$ has moment map $\mu_{\GL_n} = \sum_{i \in I} \res_i$.  Thus,
\begin{equation*}
(\prod_{i \in I} \tM(A_i)) \sslash_0 \GL_n(\C) = \mu_{\GL_n}^{-1} (0) / \GL_n(\C).
\end{equation*}

In the remainder of this section, we describe a larger moduli space in which we fix only the
combinatorics of $(V_i, \n_i)$ at $x_i$, and not the formal type.
Again, it will be constructed as a reduction of the product of local
pieces.  

First, we discuss these local pieces.  If $P$ is a uniform parahoric
subgroup $P$ with period $e$, $r\ge 0$ is an integer such that
$(r,e)=1$, and $T$ is a maximal torus such that $T(\fo)\subset P$, we
let $(\fP/ \fP^{r+1})^\vee_{\reg}$ be the set of $\g\in (\fP/
\fP^{r+1})^\vee$ such that $(P, r, \g|_{\fP^r/\fP^{r+1}})$ is a
regular stratum centralized by $T$.  We also let $\A(P, r)$ be the
subset of $(\fP/\fP^{r+1})^\vee_{\reg}$ consisting of normalized
formal types.  Note that $\A(P, r)$ is the subset of $(\fP/
\fP^{r+1})^\vee_{\reg}$ consisting of $X + H_T$, where $X$ is
stabilized by the coadjoint action of $T(\fo)$.  We also define
$(\fP^1 / \fP^{r+1})^\vee_{\reg}$ to be  the projection of $(\fP /
\fP^{r+1})^\vee_\reg$ onto $(\fP^1/ \fP^{r+1})^\vee$.

\begin{definition}\label{tmp} 
If $r = 0$, define $\tM (P, r) = \tM(A)$ for any $A$.
If $r > 0$, define $\tM(P, r) \subset (U \backslash \GL_n(\C)) \times \gl_n(\fo)^\vee$ by
\begin{equation*}
  \tM(P, r)=\{ (U g, \a) | \pi_{\fP}(\Ad^*(g) \a) \in (\fP/ \fP^{r+1})^\vee_{\reg}\}.
\end{equation*}

\end{definition}
The space $\tM(P, r)$ is a manifold; the argument is similar to the
proof that $\tM(A)$ is smooth given in \cite{BrSa}.  

Set $\fV = \pi_{\fP}^{-1} (\fP/ \fP^{r+1})^\vee_{\reg} \subset
(\fg / \fP^{r+1})^\vee$.  
In the notation of Proposition~\ref{cores} part~\eqref{cores5}, let 
$\tW_r \subset \fP^r$ 
be the subset of elements that map to $W_r \pmod{\fP^{r+1}}$.
We define $Z = (\fg / \tW_r)^\vee$. 
Note that in the case $r=0$, $\fV=(\fg/\fg^1)^\vee_{\reg}\cong(\tfl)'=\A(\GL_n(\fo),0)$
and $Z \cong \tfl$.

\begin{lemma}\label{altdescr}
  There is an isomorphism $ \fV \times_{U} \GL_n(\C) \cong\tM(P, r)$.
  Furthermore, there are open dense inclusions $(\fP/
  \fP^{r+1})^\vee_{\reg} \hookrightarrow (\fP / \tW_r)^\vee$, $\fV
  \hookrightarrow Z$, and $(\fP^1/ \fP^{r+1})^\vee_{\reg}
  \hookrightarrow (\fP^1 / \tW_r)^\vee$ in the case $r \ge 1$.
\end{lemma}
\begin{proof}
The first isomorphism is given by the map $(v, g) \mapsto (Ug, \Ad^*(g^{-1})
v)$.   Next, we observe that $(\fP/ \fP^{r+1})^\vee_{\reg}$ may be identified
with an open subset of $(\ft \cap \fP^{-r} + \fP^{-r+1}) / \fP^1$.
By part \eqref{cores4} of Proposition~\ref{cores}, 
$\tW_r = (\ft \cap \fP^{-r} + \fP^{-r+1})^\perp$.  It follows that
$(\ft \cap \fP^{-r} + \fP^{-r+1}) / \fP^1 \cong (\fP / \tW_r)^\vee$.  The other inclusions
are proved similarly.
\end{proof}

\begin{proposition}\label{A}
There is a smooth map $\mA : \tM(P, r) \to \A(P, r)$ which assigns to
$(U g, \a)$ the unique normalized formal type in the $P^1$-orbit of
$\pi_\fP(\Ad (g) (\a))$.  The fiber $\mA^{-1} (\tA)$ is isomorphic to
$\tM(A)$.
\end{proposition}
\begin{proof}
Using the description of $\tM(P, r)$ in Lemma~\ref{altdescr}, we will
construct a smooth map $\z : \fV \times \GL_n(\C) \to \A(P, r)$
and then show that it $U$-equivariant.

First, we show that 
\begin{equation}\label{regprod}
(\fP/\fP^{r+1})^\vee_{\reg}\cong\A(P, r)
\times_{T\cap P^1} P^1/ P^{r+1}, 
\end{equation}
where $T\cap P^1$ acts on the right
factor by left translation, and on the left factor by the coadjoint
action.  There is a natural map $\A(P, r) \times_{T \cap P^1} P^1/P^{r+1}
\to (\fP/\fP^{r+1})^\vee_{\reg}$ given by $(Y, \bp) \mapsto
\Ad^*(p^{-1}) Y$.  The inverse map takes a regular functional $\g$ to
the class of $(A,\bp)$, where $\tA$ is the unique normalized formal type in the
$P^1$-orbit of $\tA$ and $\tA=\Ad^*(p)\g$. 

We now define $\z' : (\fP/\fP^{r+1})^\vee_{\reg} \to \A(P, r)$ as the
projection onto the left factor of $\A(P, r) \times_{T\cap P^1} P^1 /
P^{r+1}$; this makes sense since the coadjoint action of $T\cap P^1$
on $\A(P, r)$ is trivial by Proposition~\ref{stab}.  In
particular, it is clear that $\z' (\Ad^* (u) X) =X$ for any $u
\in U$.  It follows that the map $\z (v,g) = \z' (\pi_{\fP} (v))$
is $U$-equivariant, where $U$ acts trivially on $\A(P, r)$.
We define $\mA$ to be the map induced by $\z$ on $ \fV\times_{U}
\GL_n(\C)$.

Finally, we see that $\tM(A)$ embeds into $\mA^{-1} (\tA)$ by comparing
Definitions~\ref{tM} and \ref{tmp}.  Moreover, if $(U g, \a) \in
\mA^{-1} (\tA)$, then $\pi_{\fP^1}(\Ad^*(g) (\a)) \in \Oo$.
Therefore, $\tM(A) \cong \mA^{-1} (\tA)$.

\end{proof}

Before proceeding, we need to recall some facts about Poisson
reduction.  Recall that if a Lie group $G$ acts on a Poisson manifold
$M$ via a canonical Poisson action, then there is a corresponding
moment map $\mu_M : M \to \fg^\vee$.  The following result appears in
\cite{MR}.


\begin{proposition}[{\cite[Examples 3.B, 3.F]{MR}}]\label{plem2}
  Let $M$ be a Poisson manifold, and suppose that $G$ is a linear
  algebraic group with a free canonical Poisson action on $M$.  If $0$
  is a regular value for $\mu_M$, then the Poisson structure on $M$
  induces a Poisson structure on $M \sslash_0
  G\overset{\mathrm{def}}{=}\mu^{-1} (0) / G$ called the Poisson
  reduction of $M$.  If the symplectic leaves of $M$ intersect the
  $G$-orbits cleanly (in the terminology of \cite[II.25, p.
  180]{GuSt}), then the symplectic leaves of $M \sslash_0 G$ are the
  connected components of the symplectic reductions of those
  symplectic leaves of $M$ that intersect $\mu^{-1}(0)$.
\end{proposition}
\begin{lemma} \label{regpoisson} When $r \ge 1$, $(\fP^1/
  \fP^{r+1})^\vee_{\reg}$ has a natural Poisson structure.
\end{lemma}

\begin{proof} By Lemma~\ref{altdescr}, $(\fP^1/ \fP^{r+1})^\vee_{\reg}$
is naturally isomorphic to an open dense subset of $(\fP^1 / \tW_r)^\vee$.
  It suffices to show that $(\fP^1 / \tW_r)^\vee$ is a Poisson manifold.  Since
  $[\fP^1, \tW_r] \subset \fP^{r+1}$, and the Poisson bracket restricted
  to linear functions on $(\fP^1 / \fP^{r+1})^\vee$ is just the usual
  Lie bracket on $\fP^1 / \fP^{r+1}$, this implies that the ideal
  generated by $W_r \subset \fP^1 / \fP^{r+1}$ is a Poisson ideal.  
  It follows that 
  $ (\fP^1 / \tW_r)^\vee \cong W_r^\perp \subset (\fP^1 /
  \fP^{r+1})^\vee$ is a Poisson space.  Therefore $(\fP^1 /
  \fP^{r+1})^\vee_{\reg}$ is Poisson.
\end{proof}

\begin{proposition}\label{plemma}
The manifold $\tM(P, r)$ has a Poisson structure.
When $r \ge 1$, 
the manifold $\tM(P, r)$ is isomorphic to a Poisson reduction:
\begin{equation*}
\tM(P, r) \cong ((\fP^1/ \fP^{r+1})^\vee_{reg} \times T^* \GL_n(\C))  \sslash_{0} U.
\end{equation*}
The symplectic leaves of $\tM(P,r)$ are the 
 fibers of the map $\mA$.
\end{proposition}

\begin{proof}
  In the case $r = 0$, $\tM(P, r) \cong \tM(A)$ by the remark
  following Definition~\ref{tmp}.  Therefore, $\tM(P, r)$ is in fact
  symplectic.
  
  We now suppose that $r \ge 1$.  The space $(\fP^1/
  \fP^{r+1})^\vee_{reg} \times T^* \GL_n(\C))$ is a Poisson manifold
  using the Poisson structure of Lemma~\ref{regpoisson} on the first
  factor and the usual symplectic structure of a cotangent bundle on
  the second.
  

  The group $U$ acts on $(\fP^1 / \fP^{r+1})^\vee_\reg$ and
  $T^* \GL_n(\C)$
  via the coadjoint action and the free action induced by
  left multiplication on $\GL_n(\C)$ respectively; these actions are
  canonical Poisson.

The moment map of the diagonal action is given by
\begin{equation}\label{tmu}
\tmu (Y, (g, X)) = -\Ad^*(g) (X)|_{\fu} + Y |_{\fu}.
\end{equation}
It is clear that $\tmu$ is a submersion.


If $\tmu (Y, (g, X)) = 0$, then $\Ad^* (g) (X) |_{\fu} = Y |_{\fu}$.
Therefore, we may glue $\Ad^* (g) (X)$ and $Y$ together to obtain a functional
$\psi_{Y, X} \in (\fg / \fP^{r+1})^\vee$; note that $\psi_{Y,
  X}\in\fV$ if and only if $Y\in(\fP^1/ \fP^{r+1})^\vee_{reg}$.
Using the description of $\tM(P, r)$ in
Lemma~\ref{altdescr},  we define a map $p : \tmu^{-1} (0) \to \tM(P, r)$ by
\begin{equation*}
  (Y, (g, X)) \mapsto (\psi_{Y, X}, g) \in (\fg / \fP^{r+1})^\vee \times_{U} \GL_n(\C).
\end{equation*}
The map is surjective, and the fibers of $p$ are $U$-orbits.

The symplectic leaves of $(\fP^1 / \fP^{r+1})^\vee_{reg} \times T^*
\GL_n(\C)$ are given by $\O \times T^*\GL_n(\C)$, where $\O$ is any
coadjoint orbit in $(\fP^1/ \fP^{r+1})^\vee_{reg}$.  It is obvious
that the $U$-orbits intersect the leaves cleanly.  It now follows from
Definition~\ref{tM} and the fact that the $\tM(A)$ are connected that
the symplectic leaves of $\tM(P, r)$ are given by $\tM(A)$ for $A$ a
formal type corresponding to $P$ and $r$.

\end{proof}

\begin{lemma}\label{faction} The $\GL_n(\C)$-action on $\tM(P, r)$ defined by $h
  (U g, \a) = (U g h^{-1}, \Ad^*(h) (\a))$ is free canonical Poisson
  with submersive moment map $\mu(Ug,\a)=\res(\a)$.
\end{lemma}
\begin{proof} By \cite[Lemma 5.12]{BrSa}, this action restricts to a
  free action on each symplectic leaf $\tM(A)$.  To see that it is
  canonical Poisson, first observe that the $\GL_n(\C)$-action on
  $(\fP^1/ \fP^{r+1})^\vee_{reg} \times T^* \GL_n(\C)$ given by
  $h\cdot(Y,(g,X))=(Y,gh^{-1},Ad^*(h)X)$ is canonical Poisson with
  moment map $(Y,(g,X))\mapsto X$.  Since it commutes with the action
  of $U$, it induces a canonical Poisson action on the Poisson
  reduction with moment map given by the same formula.  It is easy to
  check that this action corresponds to the given action on $\tM(P,r)$
  under the isomorphism of Proposition~\ref{plemma}.  Since
  $(Y,(g,X))$ corresponds to $(\psi_{Y,X},g)$ and
  $\res(\psi_{Y,X})=X$, we obtain the desired expression for the
  moment map.  Finally, by \cite[Lemma 5.11]{BrSa}, $\mu$ is even a
  submersion when restricted to any symplectic leaf.
\end{proof}

We are now ready to construct the moduli space of framed connections
on  $\proj(\C)$ with fixed combinatorics.    Recall that $\bfx$ is a
finite set of points in $\proj(\C)$ indexed by $I$.  Let
$\bfP = \{P_i\}_{i \in I}$ be a collection of uniform parahoric
subgroups with periods $e_i$ such that $P_i \in \GL_n(\fo_i)$, and let  $\bfr
= (r_i)_{i \in I}$ with $r_i \ge 0$ and $\gcd(r_i, e_i) = 1$.  Also,
fix maximal tori $T_i$ such that $T_i(\fo)\subset P_i$.

It is immediate from Lemma~\ref{faction} that the diagonal action of
$\GL_n(\C)$ on $\prod_{i \in I} \tM(P_i, r_i)$ is free canonical
Poisson and that its moment map $\mu=\sum_{i\in I}\res_i$ is a
submersion.
\begin{definition}\label{tMg}
 Define $\tM(\bfx, \bfP, \bfr)$ as the Poisson reduction
\begin{equation*}
\tM(\bfx, \bfP, \bfr)=\prod_{i \in I} \tM(P_i, r_i) \sslash_0 \GL_n(\C).
\end{equation*}
We also set $\tMg(\bfx, \bfP, \bfr)=\mu^{-1}(0)$ and $\A(\bfx, \bfP,
\bfr)=\prod_{i \in I} \A(P_i, r_i)$.
\end{definition}

\begin{proposition}\label{mma}
  There is a smooth map $\mmA: \tM(\bfx, \bfP, \bfr) \to \A(\bfx,
  \bfP, \bfr)$ which assigns a normalized formal type $\tA_i$ to $\n$
  at each pole $x_i$.  The fiber $\mmA^{-1} (\tilde{\bfA})$ is isomorphic to
  $\tM(\bfA)$, and the symplectic leaves of $\tM(\bfx, \bfP, \bfr)$ are
  the connected components of these fibers.
\end{proposition}
\begin{proof} Since the maps $\mA_i : \tM(P_i, r_i) \to \A(P_i, r_i)$
  are $\GL_n(\C)$-equivariant (where $\GL_n(\C)$ acts trivially on
  $\A(P_i, r_i)$), $\prod_{i \in I} \mA_i$ induces the desired map
  $\mmA$. The statement about the fibers of $\mmA$ follows from
  Proposition~\ref{A} and the construction of $\tM(\bfA)$ in
  Theorem~\ref{modthm}.  Finally, the symplectic leaves of $\prod_i
  \tM(P_i, r_i)$ are given by $\prod_i \tM(A_i)$, and they  
intersect the $\GL_n(\C)$-orbits cleanly.  By Proposition~\ref{plem2}, 
the symplectic leaves of $\tM( \bfx, \bfP, \bfr)$ are the connected
components of  $\tM(\bfA)$.
\end{proof}

\begin{theorem}\label{modspace}
  The Poisson 
manifold $\tM (\bfx, \bfP, \bfr)$ is isomorphic to the moduli
  space of isomorphism classes of triples $(V, \n, \bfg)$, where $(V,
  \n, \bfg)$ satisfies the first two conditions of Definition
  \ref{tmdef}, and $(V_i, \n_i)$ contains a regular stratum of the
  form $(P_i, r_i, \b)$.  The manifold $\tMg (\bfx, \bfP, \bfr)$ is
  isomorphic to the moduli space of isomorphism classes $(V, \n,
  \bfg)$ satisfying the conditions above, with a fixed global
  trivialization $\phi$.
\end{theorem}
\begin{proof}
By Proposition~\ref{uniqft}, we may associate
a unique formal type $A$ to every formal connection that contains
a regular stratum $(P, r, \b)$.  Therefore, if $(V, \n, \bfg)$ satisfies the conditions above,
there is a unique element $\bfA \in \A(P_i, r_i)$ given by the formal type of $(V, \n, \bfg)$
at each singular point.  In particular, by Theorem~\ref{modthm}, $(V, \n, \bfg)$
corresponds to a unique point in $\tM(\bfA)$.  However,
by Proposition~\ref{mma}, $\tM(\bfA) \cong \mmA^{-1} (\bfA)$.  On the other hand, 
every point  $p \in \tM(\bfx, \bfP, \bfr)$ corresponds to a unique
connection with formal type $\mmA(p)$.

Now, suppose $m = (U_i g_i, \a_i)_{i \in I} \in \tMg (\bfx, \bfP,
\bfr)$.  Fix a global form $\nu \in \Om^1_{\proj}$.  By \eqref{pair},
we may associate to $\a_i$ a unique meromorphic form $\a_{i \nu} \nu$
with coefficients in $\gl_n(F_i)$.  Since $\mu(m) = 0$, $\sum_{i \in
  I} \Res_i (\a_{i \nu} \nu) = 0$.  It follows that $q$ determines a
global form $\a \nu$, and thus a global meromorphic connection $\n = d
+ \a \nu$ on the trivial rank $n$ vector bundle over $\proj$.  It is
easily checked that this gives a bijection between points in $\tMg
(\bfx, \bfP, \bfr)$ and triples $(V, \n, \bfg)$.

\end{proof}

\section{Integrable deformations}\label{four}
Let $X = \proj(\C)$, and let $V$ be an $n-$dimensional trivial vector
bundle on $X$.  In this section, we will consider the deformations of
a connection $(V,\n)$ which contains a regular stratum at each
singularity.

\subsection{Formal  Deformations}\label{4.1}
Without loss of generality, take $x_1 = 0$ and fix a parameter $z$ at
$0$.  We will suppress the subscripts on $(P_i, r_i, \b_i)$, etc. when
we work locally at $0$.  Let $F$ be the field of Laurent series at $0$
and $\fo \subset F$ the ring of power series.  Define $\hV$ to be the
formal completion of $V$ at $0$, and let $\hn$ be the induced formal
connection.

Now, let $\Df = \Spec (\fo)$, and fix a standard parahoric $P$ and an
integer $r$ with $\gcd(r, e) = 1$.  We also fix a torus $T$, with
$T(\fo) \subset P$, as in \eqref{torus}.  Let $\De$ be an analytic
polydisk; we denote its ring of functions by $R$.  A \emph{formal flat
  deformation} is a flat, meromorphic connection $(\bar{V}, \bn)$ on
$\Df \times \De$ satisfying the following properties:
\begin{itemize}
\item the vector bundle $\bar{V}$ is isomorphic to the trivial rank
  $n$ vector bundle, and
\item  the restriction of $\bn$
to $\Df \times \{y\}$, denoted by $\hn_y$, contains a regular  stratum $(P_y, r, \b_y)$.  
\end{itemize}  
Fix a trivialization $\phi$ of $\bar{V}$ so that we may identify all
other trivializations with elements of $\GL_n(\fo\otimes R)$.  We say
that the deformation $(\bar{V}, \bn)$ is \emph{framed} if there exists
a trivialization $g \in \GL_n (R)$ with the property that $P_y =
P^{g(y)}$, $\b_y = \b^{g(y)}$, and the regular stratum $(P_y, r,
\b_y)$ is centralized by $T_y:=T^{g(y)}$.  In particular, any
representative $(\b_y)_\nu \in \fP_y^{-r}$ for $\b_y$ lies in $\ft_y +
\fP_y^{-r+1}$ by \cite[Remark 3.5]{BrSa} (with $\ord(\nu)=-1$).

We denote $\fPD^\ell = \fP^\ell \otimes R,$ $\bfPD^\ell =
\fPD^\ell/\fPD^{\ell+1}$, and $\ftD = \ft \otimes R$.  We define $\PD$
and $\TD$ similarly.  Suppose that $(\bar{V}, \bn)$ is framed by $g$,
and fix a nonzero one-form $\nu$ at $0$.  If $\nu = u \frac{dz}{z},$
write $\frac{z \nu}{d z} = u$.  Let $A (y)$ be the formal type of
$\bn$ at $y$.  Using the pairing in \eqref{pair}, we may choose a
representative $A_\nu(y)$ for $A(y)$ of the form
\begin{equation*}
A_\nu (y) = (\frac{-r}{n} t_{-r}(y) \ot^{-r} + \ldots + \frac{-1}{n} t_{-1}(y) \ot^{-1} + t_0(y))\frac{dz}{z \nu},
\end{equation*}  
with $t_i(y) \in \tfl_\De$.  For example, $A_{\frac{dz}{z}} =
(\frac{-r}{n} t_{-r}(y) \ot^{-r} + \ldots + \frac{-1}{n} t_{-1}(y)
\ot^{-1} + t_0(y))$.

Recall that any element of $\ft$ can be written as a Laurent series
$t = \sum_{i = -N}^\infty t_j \ot^j$ with $t_i \in \tfl$.
We define an endomorphism $\delta_e$ of $\ft$ via $\delta_e(t) =  \sum_{i = -N}^\infty \frac{i}{e} t_i \ot^i$.
\begin{lemma}\label{del}  Suppose $t \in \ft_{\De}$.
Then, $z \pd_z t - [t, H_T]= \d_e t$.  Moreover,
any solution $B \in \Om^1_\De((\ft + \fP^\ell)/ \fP^\ell)$ to the differential equation 
$ z \pz B - [B, H_T]= \delta_e t + t_0 + \fP^\ell$ has the form
$B = t + f+ \fP^\ell$, where $f\in \tfl_\De$.  There is no solution
when $\ell \ge 0$ and $t_0 \ne 0$.
\end{lemma}
\begin{proof}
  Since the equations are block diagonal in $\ft$, we may immediately
  reduce to the case where $T = E^\times$.  A direct calculation shows
  that $z \pz \ot^{i} - [\ot, H_E] = \frac{i}{ e} \ot^i = \delta_e
  (\ot^i)$.  This proves the first statement.  The second follows by
  applying the same calculation to each term of $B$ up to $\ot^\ell$.
  Note that $z \pz \ot^0 - [\ot^0, H_E] = 0$, so there is no solution
  when $t_0 \ne 0$.
\end{proof}

As in the previous section, we will let $p \cdot [\bn]_g$ denote the
$\dz$ part of the gauge transformation formula: $p \cdot [\bn]_g =
\Ad(p) ([\bn]_g) - (\dz p)p^{-1}$.  Similarly, upon fixing $\nu \in
\Om^1_{F} (R)$, we write $p \cdot M = \Ad(p) (M) - (\tau_\nu p)
p^{-1}$ when $M \in \gl_n (F \otimes_\C R)$.  We will use $\dy$ and
$\bd$ to denote the exterior differential on $\De$ and $\proj \times
\De$, respectively.  Suppose that $\tA_\nu$ is the normalized formal
type associated to $A_\nu$.  By Proposition~\ref{gA}, there exists an
element $p \in P_\De$ such that
\begin{equation*}
p(y) \cdot [\hn_y]_{g } = \tA_\nu(y) \nu.
\end{equation*}
Let $\dy$ be the exterior differential on $\De$.  Define
\begin{equation*}
A_\De(y) = \sum_{i = -r}^{-1}  \ot^{-i} \dy t_i,
\end{equation*} 
so that $\dy A_{\frac{dz}{z}} = \delta_e A_\De  + \dy t_0$.

\begin{proposition}\label{totalconnection} Let $p$ be as above.
  Then, $p \cdot [\bn]_g - (\dy p) p^{-1} = \tA_\nu \nu + A_\De + f $,
  where $f \in \Om^1_{\De} (\tfl)$ is closed.  Moreover, $t_0 (y)$
  must be constant.  Conversely, if $q \in \GL_n(F\otimes R)$, $f$ is
  closed, and $t_0 (y)$ is constant, then $[\bn'] = q^{-1} \cdot
  \left(\tA_\nu \nu + A_\De +f \right) + q^{-1} (\dy q)$ determines a
  flat meromorphic connection $\bn'$ on $\bV$.
\end{proposition}
\begin{rmk}  The case $e=1$ is proved in the Appendix of \cite{Boa}.
\end{rmk}
\begin{proof}
Note that the connection determined by $p \cdot [\bn]_g - (\dy p) p$ is flat,
since $\bn$ is flat and $-(\dy p) p$ is simply the $\dy$
 part of the gauge transformation formula.  
Conversely, if the connection determined by 
$\tA_\nu \nu + A_\De + f$ is flat, then $\bn'$ is flat by the same
argument.

Without loss of generality, set $\nu = \frac{dz}{z}$.
It suffices to show that whenever
$[\bn] = \tA_{\frac{dz}{z}} \frac{dz}{z} +  B $ for some 
$B \in \Om^1_\De (\gl_n(F)),$ then $\bn$ is flat if and only if
$B$ has the form $ A_\De + f$.
We observe that $\bn$ is flat if and only if it satisfies the conditions
\begin{gather}
    d_\De (A_{\frac{dz}{z}}) - z \pz B + [B, A_{\frac{dz}{z}}+H_T] =
    0\label{flata}\\ \text{ and }\ \dy B+B\wedge B=0.\label{flatb}
\end{gather}
If $B = A_\De + f$, \eqref{flatb} holds trivially while \eqref{flata}
follows from the first part of Lemma~\ref{del}, so $\bn$ is flat.

We now prove the converse.  If $r = 0$, then $e = 1$ and $H_T = 0$.
In this case, we may take $A_{\frac{dz}{z}} = t_0 (y)$ to be a regular
diagonal matrix with entries in $R$; moreover, no two eigenvalues of
$A_{\frac{dz}{z}}$ differ by an integer.  Setting $B=\sum
B_{\ell}z^{\ell}$ with $B_{\ell}\in\gl_n(\C)$, \eqref{flatb} reduces
to $[B_\ell, A_{\frac{dz}{z}}] = \ell B_\ell$ when $\ell \ne 0$.  The
eigenvalue condition now implies that $B_\ell = 0$.  whenever $\ell
\ne 0$, by the condition on the eigenvalues of $A_{\frac{dz}{z}}$.  On
the other hand, $[B_0, A_{\frac{dz}{z}}] = -\dy (A_{\frac{dz}{z}})$.
Since the right hand side is a diagonal matrix, both sides must be
identically $0$.  It follows that $\dy (A_{\frac{dz}{z}}) = 0$ and
$B_0 \in \Om^1_{\De}(\tfl)$ while the fact that $B_0$ is closed
follows from \eqref{flatb}.

We now consider the case $r \ge 1$.  In the following, let $t_0 (y)$
be the constant term of $A_{\frac{dz}{z}}$.  Suppose, by induction,
that $B \in A_\De + f+ \Om^1_\De(\tfl+\fP^{\ell})$.  (Note that $B \in
\Om^1(\fP^L)$ for some $L \le -r$, so the first inductive step is
trivially satisfied for $\ell = L $).  Applying the first part of
Lemma~\ref{del} with $t=A_\De$ gives
\begin{equation*}
  d_\De (A_{\frac{dz}{z}}) +  [B, H_T] -  z \pz B \in \dy t_0+ \Om^1_\De
  (\fP^\ell).
\end{equation*}
We deduce from \eqref{flata} that $[B, A_{\frac{dz}{z}}] \in \dy t_0 +
\Om^1_\De(\fP^{\ell})$.

When $\ell < 1$, $\dy t_0 \in \Om^1_{\De} (\fP^{\ell})$.  This implies
that $[B, A_{\frac{dz}{z}}] \in \Om^1_\De (\fP^{\ell})$, so $B \in
\Om^1_{\De} (\ft + \fP^{\ell+r})$.  Next, consider $\ell =1$.  By the
$\ell=0$ step, we know that $B \in \Om^1_{\De} (\ft + \fP^{r})$.  Part
(3) of Proposition~\ref{cores} shows that $[B, A_{\frac{dz}{z}}] +
\Om^1_{\De} (\fP^1) \in \ker(\bpit\otimes \id_{\Om^1_{\De}})$.  Since
$\dy t_0 \in \Om^1_{\De} (\tfl),$ part (1) of the same proposition
gives $\dy t_0 \in \Om^1_{\De} (\fP^1)$.  Since $\tfl \cap \fP^1 =
\{0\}$, we see that $\dy t_0 = 0$.  Thus, we may conclude that the
inductive hypothesis implies that $[B, A_{\frac{dz}{z}}] \in \Om^1_\De
(\fP^{\ell})$ for $\ell \ge 1$ as well.  As before, $B \in \Om^1_{\De}
(\ft + \fP^{\ell+r})\subset \Om^1_{\De} (\ft + \fP^{\ell+1})$.

To complete the inductive step,
we apply the second part of Lemma~\ref{del} to obtain
$B \in A_\De +\Om^1_\De(\tfl+\fP^{\ell+1})$.  Using the fact that the
sum $\tfl+\fP^\ell$ is direct for $\ell\ge 1$, we actually obtain
$f\in\Om^1_\De(\tfl)$ such that $B \in A_\De +f+\Om^1_\De(\fP^{\ell})$
for all $\ell$.  The result now follows since $\bigcap_\ell\fP^\ell=0$.




\end{proof}
\begin{definition}
We say that a compatible framing $g$ for $\bn$  is \emph{good} if 
there exists $p \in \PD^1$  such that 
$p \cdot  [\bn]_{g} - (\dy p) p^{-1} = \tA_\nu \nu+ A_\De. $
\end{definition}
\begin{proposition}\label{good}
Every framed flat deformation has a good compatible framing.
\end{proposition}
\begin{proof}
  If $g$ is a compatible framing, there exists $p \in P$ such that $p
  \cdot ([\bn]_{g }) - (\dy p) p^{-1} = (A_{\frac{dz}{z}} + H_T)
  \frac{dz}{z} + A_\De+ f$ by Proposition~\ref{totalconnection}.
  Since $f$ is closed on $\De$, it is exact.  Choosing
  $\varphi\in\tfl_{\De}$ such that $f = \dy \varphi$, we obtain
\begin{equation*}
\Ad(e^{\varphi}) \left(p \cdot ([\bn]_{g }) - \dy(p) p^{-1}\right) - \dy \varphi
  = (A_{\frac{dz}{z}} + H_T
)\frac{dz}{z} + A_\De.
\end{equation*}
It follows that $e^{\varphi} g$ is a good compatible framing.
\end{proof}


\subsection{Global Deformations}
\begin{definition}\label{deformation}
  A \emph{framed global deformation} is a triple $(\bfg, \bV, \tn)$
  consisting of:
\begin{enumerate}
\item a trivializable rank $n$ vector bundle
$\bV$ on $\proj \times \De$;
\item  an $R$-relative connection $\tn$;
\item a collection of analytic framings $\bfg = (g_i)_{i \in I}$, $g_i
  : \De \to U_i \backslash \GL_n(\C)$;
\item the restriction of $(\bV, \tn)$ to $\proj \times \{y\}$ must lie
  in $\tM (\bfx, \bfP, \bfr)$ with compatible framing
  $\bfg(y)$.\end{enumerate} We say that a framed deformation is
\emph{integrable} if there exists a flat $\C$-relative connection $\bn$ on
$\proj \times \De$ with $\proj$ part $\tn$.
\end{definition}
We note that $(\bfg, \bV, \tn)$ determines a smooth map 
$\Delta \to \tM (\bfx, \bfP, \bfr)$.
Specifically, there are maps  $g_i (y)$ and $\a_i (y)$ such that
the connection on the fiber above $y$ corresponds to the point
$(U_i g_i(y) , \alpha_i (y))_{i \in I} \in \tM (\bfx, \bfP, \bfr)$ .

Suppose that $(\bfg, \bV, \tn)$ is an integrable framed global
deformation.  If we fix a trivialization for $\bV$, we may write
$[\bn] = \a \nu + \Up$, where $\Upsilon$ is a section of $
\Om^1_{\De/\C}(\End(\bV))$ with poles along $\{x_i\}$.  The curvature
of $\bn$ is given by
\begin{equation*}
\Xi (\a, \Up) = \bd (\a \nu + \Up) + \Up \wedge \a \nu + \a\nu \wedge \Up
+ \Up \wedge \Up
\in \Omega^2_{\De \times \proj}.
\end{equation*}
Thus, $\bn$ is flat if and only if the following hold:
\begin{equation}\label{flatness}
  \tau_\nu \Up  = d_\De \a +  [\Up, \a]\quad \text{ and}
  \quad 0  = d_\De \Up +\Up \wedge \Up.
\end{equation}


An integrable deformation $(\bfg, \bV, \tn)$ determines a flat formal
deformation $(\hV_i, \bn_i)$ at each singular point.  Therefore, if
$\tA_i(y)$ is the normalized formal type of $(\hV_i, \bn_i)$ at $y \in
\De$, Proposition~\ref{good} implies that there exists $p_i \in P^1_i$
such that $p_i g_i \cdot [\bn] - \Ad(p_i) (\dy g_i g_i^{-1}) - (\dy
p_i)p_i^{-1}= \tA_{i, \nu}\nu+ A_{i, \De}.$ Since $p_i^{-1} d_\De p_i
\in \Om^1_\De(\fP^1)$ by \cite[Lemma 4.4]{BrSa}, we deduce
\begin{equation}\label{Omod}
  \Up \in \Ad(g_i^{-1} p_i^{-1}) A_{i, \De} + g_i^{-1} d_\De g_i + \Om^1_\De((\fP_i^1)^{g_i}).
\end{equation}





Set $\bfre = (r_i')_{i  \in I} := (\lceil \frac{r_i}{e_i} \rceil)_{i \in I}$, where $\lceil \frac{r_i}{e_i} \rceil$
is the integer ceiling of $\frac{r_i}{e_i}$.  Let $D_{\bfr}$ be the
divisor $\sum_{i\in I} r_i'[x_i]$ on $\proj$.

\begin{definition}\label{spgr}
  Let $\sg$ denote the trivial $\gl_n$-bundle on $\proj$, and let
  $\sgr=\sg(D_{\bfr})$ be the sheaf corresponding to the divisor
  $D_{\bfr}$
\end{definition}
Note that sections of $\sgr$ have poles of order at most $r_i'$ at
$x_i$.

Fix a set of parameters $(z_i)_{i \in I}$ at each singular point $x_i$
with the property that each $z_i$ has a pole at a fixed point $x_0$.
Define a $\C$-linear map $\phi_i : \gl_n (F_i) \to \gl_n(\C)$ by
\begin{equation}\label{bphi}
\phi_i (X) = \Res_{x_i} (X \frac{dz_i}{z_i}).
\end{equation}  
Thus, $\phi_i$
extracts the constant term of $X$ at $x_i$ with respect to $z_i$.
This induces a map $\bphi_i : \gl_n(F_i)/ \fP^1_i \to \gl_n(\C) /
\fu_i$.

Next, we define a map $\vs$ which assigns a global section of $\sgr$
to a collection of principal parts at $\bfx$.  Given $X_i\in
\fg^{-r'_i}_i / \fg_i$, let $\tilde{X}_i\in H^0(\proj;\sgr)$ be the
section corresponding to the unique lift of $X_i$ to
$\gl_n(z_i^{-1}\C[z_i^{-1}])$.

\begin{definition}\label{Om0}
The map $\vs:\prod_{i \in I} \fg^{-r'_i}_i / \fg_i\to H^0 (\proj;
\sgr)$ is given by \begin{equation}\label{vs}
\vs( (X_i)_{i \in I}) = \sum_{i \in I} \tilde{X}_i.  
\end{equation}
We will usually write $X^0$ for $\vs( (X_i)_{i\in I})$.
\end{definition}

\begin{rmk} This map commutes with the adjoint action of $\GL_n(\C)$,
  i.e., $\Ad(g)(X^0)=(\Ad(g)X)^0$ for any $g\in\GL_n(\C)$.  Indeed,
  $\Ad(g)(\tilde{X}_i)=\widetilde{\Ad(g)X_i}$ for each $i$ since
  $\Ad(g)$ stabilizes $\gl_n(z_i^{-1}\C[z_i^{-1}])$.
\end{rmk}

We are now ready to describe a system of differential equations satisfied by an integrable 
deformation.
Let $(\bV, \bn)$ be a deformation of $(V, \n)$ as in Definition \ref{deformation}, corresponding
to a map $(g_i(y), \alpha_i(y))_{i \in I}$ from $\Delta$ to $\tM (\bfx, \bfP, \bfr)$.

Fix a uniformizer $z_0$ at $x_0$. 
\begin{lemma}\label{res0}
  If $\Up'$ has principal part $\Up_i$ at each $x_i$ and is
  holomorphic elsewhere, then $\Up'- \Up^0 = \Up'|_{z_0 = 0} \in
  \Om^1_{\De}(\gl_n(\C)).$ Here, $\Up^0 = \vs ((\Up_i)_{i \in I})$ as
  above.
\end{lemma}

\begin{proof}
 Since $\Up'$ and $\Up^0$ have the same
  principal parts at each singular point, $\Up' - \Up^0 = X \in
  \Om^1_{\De} (\gl_n (\C))$.
By construction, $\Up^0$ and $\Up'$ are holomorphic at $x_0$, 
and  $\Up^0|_{z_0 = 0} = 0$.  
Therefore, $X = \Up' |_{z_0 = 0}$.
\end{proof}
 \begin{lemma}\label{Omzero}
   Given $\Up \in \Om^1_{\De} (H^0 (\proj; \sgr))$, define $\Up_i =
   \Up + \fg_i$.  Any integrable framed deformation $\bn$ is
   $GL_n(R)$-gauge equivalent to a deformation $ \a \nu + \Up $
   satisfying $\Up = \vs ((\Up_i)_{i \in I})$.
\end{lemma}
\begin{proof}

  Suppose that $[\bn] = \a' \nu + \Up'$.  Then, $\Up' - (\Up')^0 = X
  \in \Om^1_\De (\gl_n(\C))$ by Lemma~\ref{res0}.  Since the image of
  $\vs$ is closed under conjugating by $\GL_n(R)$, it suffices to show
  that there exists $g \in \GL_n(R)$ such that $\Up:=\Ad (g) \Up' -
  (\dy g) g^{-1} = \Ad (g) (\Up')^0$.

  The system of differential equations $g^{-1} (\dy g) = X$ has a
  solution if $\dy X + X \wedge X= 0$.  To see this, recall that $\dy
  X = \dy \Up'|_{z_0 = 0}$.  By flatness, $\dy \Up' + \Up' \wedge \Up'
  = 0$, so $\dy X + X \wedge X = (\dy \Up' + \Up' \wedge \Up' )|_{z_0
    = 0} = 0$.
\end{proof}

Now, suppose that $(\bfg,\bV, \tn)$ is a deformation.  Choose $p_i \in
P_\De^1$ such that $p_i \cdot \a \nu = \tA_{i \nu} \nu$, and write
$\hg_i = p_i g_i$.  Let $\Up_i = \Ad(\hg_i^{-1}) A_{i, \De}+
\Om^1_{\De} (\fg_i)$, with corresponding global section $\Up^0\in
\Omega^1_{\De}(H^0(\proj;\sgr))$.
\begin{theorem}\label{defeq}
A good framed deformation $( \bfg, \bV, \tn)$ is integrable if and only if it is $\GL_n(R)$-gauge-equivalent to a deformation satisfying the following equations:
\begin{enumerate}
\item\label{flow} 
$\phi_i (\Ad (p^{-1}_i) (A_{i,\De}) +  (d_\De g_i) g_i^{-1}) \in \phi_i (\Ad(g_i)\Up^0) + \Om^1_{\De} (\fu_i)$;
\item\label{curv} $d_\De \a \in 
\tau_\nu \Up^0 + [\a, \Up^0 ] + \Om^1_{\De} (\fg_i \frac{dz_i}{z_i
  \nu});$ and
\item\label{dy}  $d \Up^0 + \Up^0 \wedge \Up^0=0$.
\end{enumerate}
\end{theorem}

\begin{proof}
  Suppose that $(\bfg, \bV, \tn)$ is integrable.  By Lemma
  \ref{Omzero}, there exists a gauge $g\in\GL_n(R)$ such that
  $[\bn]_{g} = \a \nu + \Up$ with $\Up = \Up^0$.  
  By Proposition \ref{totalconnection},
\begin{equation}\label{omcondition}
\Up \in \Ad (\hg_i^{-1}) A_{i, \Delta} + \hg_i^{-1} d_\De \hg_i + (\fP^1_{i, \De})^{g_i}.
\end{equation}
Therefore, condition \eqref{flow} is satisfied by applying $\Ad(g_i)$
and $\phi_i$.  Moreover, $\a \nu + \Up$ is flat, so conditions
\eqref{curv} and \eqref{dy} follow from \eqref{flatness}.


To see the converse, note that condition~\eqref{flow} shows that
$\Up^0$ satisfies \eqref{omcondition}.
Condition~\eqref{curv} implies that the cross-term of the curvature of
$\alpha \nu + \Up^0$ is zero modulo $\fg^1_i \frac{dz_i}{z_i
  \nu}$.  Writing this term as $\sum f_j dy_j$, where the $dy_j$'s are
a basis for $\Om^1_\De$, we see that $f_j\in H^0 (\proj;
\Om^1_{\proj}) = \{0\}$.  Since the term in $\Om^2_\De$ vanishes by
condition~\eqref{dy}, we see that the curvature vanishes. 

\end{proof}
\begin{corollary}\label{eqdfq}
A good framed deformation $(\bfg, \bV, \tn)$ is integrable if and only if there exists $g \in \GL_n (R)$
such that 
\begin{enumerate}
\item\label{eqflow} $\phi_i(\Ad(g_i)(g^{-1} \dy g  + \Up^0)) 
\in  \phi_i ( \Ad (p_i^{-1}) A_{\De, i} + (\dy g_i) g_i^{-1}) +\Om^1_{\De} (\fu_i)$;
\item\label{eqcurve} $\dy \alpha \in   \tau_\nu \Up^0   + [\a, \Up^0 + g^{-1} \dy g] +
\Om^1_{\De} (\fg^1_{i} \frac{dz_i}{z_i \nu});$ and
\item\label{eqdO}  $\dy (\Up^0 +g^{-1} \dy g) + 
(\Up^0 + g^{-1} \dy g)\wedge (\Up^0 + g^{-1} \dy g) = 0$.
\end{enumerate}
\end{corollary}
\begin{proof} Set $[\bn]=\alpha\nu + \Up$ and choose $g\in\GL_n(R)$
  such that $[\bn]_g=\alpha'\nu + \Up'$ with $\Up'=(\Up')^0$.  The
  proof above shows that $(\bfg',\alpha',\Up')$ satisfy the conditions
  in the theorem.
  One obtains the equations in the corollary by substituting
  $\Up'=\Ad(g)\Up^0$, $\alpha'=\Ad(g)(\alpha)$, and $g'_i=g_ig^{-1}$.
The converse is proved similarly.
\end{proof}
\begin{rmk}  It will be shown later in Theorem~\ref{redundant} that
  the third condition in these two results is unnecessary.
\end{rmk}




\section{The Differential Ideal}\label{integrability}
In this section and the next, we show that the equations from
Corollary \ref{eqdfq} determine a Frobenius integrable system on
$\tM(\bfx, \bfP, \bfr)$.  Throughout, we will fix a global meromorphic
one-form $\nu$ on $\proj(\C)$.  We will also use the convention that
if $B$ is a vector bundle on a manifold $M$ and $N\subset M$ is open,
then we will write $\Om^1_N(B)$ for $\Om^1_N(B|_N)$.  Furthermore, if
$\sigma$ is a global section of $\Om^1_M(B)$, we will abuse notation
and let $\sigma$ also denote $\sigma|_N$.  
\subsection{Generators}
We will construct a differential ideal $\I$ on $\tMg(\bfx,
\bfP, \bfr)$ corresponding to the system of differential equations in
Corollary \ref{eqdfq}.  Throughout, we will simplify notation by
suppressing $\bfx$, $\bfP$, and $\bfr$ in the notation and using $d$
for the exterior differential on all spaces whenever there is no risk
of ambiguity.  We will also write $\tM_i = \tM(P_i, r_i)$ and $\A_i =
\A(P_i, r_i)$.

Set $\ft^j_i = \ft_i \cap \fP_i^j$ and $T_i^j = T_i\cap P_i^j$.
First, we define maps $\hmA_i : \tMg \to \A_i$ and $\hmA_{i \nu} :
\tMg \to (\ft_i^{-r_i} + H_T) \frac{dz_i}{z_i \nu}$: 
$
\hmA_i ( (U_j g_j,\a_j
)_{i \in I}) = \mA_i(U_i g_i,\a_i)
$
and $\hmA_{i \nu} ( (U_j g_j,\a_j))$ is the standard representative of
$\mA_i (U_ig_i,\a_i)$ in $\fP_i^{-r}\frac{dz_i}{z_i \nu}$ with Laurent
expansion $(\g_{-r} \o_i^{-r} + \ldots +\g_{0} \o_i^0 + H_T)
\frac{dz_i}{z_i \nu}$.   We also write $\hmA_i^0$ for the
coefficient of $\o_i^0 \frac{dz_i}{z_i \nu}$ in $ \hmA_{i \nu}$.

We identify the tangent bundle of $\A_i \subset (\ft_i/ \ft_i^{r+1})^\vee$
with $\A_i \times \ft_i^{-r_i} / \ft_i^1$
using the pairing $\langle,
\rangle_{\frac{dz_i}{z_i}}$.  
The differential of the map $\hmA_i :
\tMg \to \A_i$ determines a section $d \hmA_i$ of $\Omega^1_{\tMg}
(\ft_i^{-r_i} / \ft_i^1) \subset \Omega^1_{\tMg} (\fP_i^{-r_i}/\fP_i^1)$.

Identify $\tM_i$ with $\fV_j \times_{U_j} \GL_n(\C)$ as in Lemma~\ref{altdescr}.
Throughout this section, we will
use $m_j = (v_j, g_j) \in \fV_j \times_{U_j} \GL_n(\C)$ to denote a
point in $\tM_i$ and $m = (m_j)_{j \in I}$ to denote a point in $\tMg
\subset \prod_{i \in I} \tM_i$. Let $\psi_i : \tMg \to U_i \backslash \GL_n(\C)$ be the composition of
the projections $ \tMg \to \tM_i$ and $ \tpsi_i : \tM_i \to U_i \backslash
\GL_n(\C)$. 

We are now ready to define a collection of differential forms on
$\tMg$.

\begin{enumerate}
\item 
Define an endomorphism of $\ft_i$ by
\begin{equation*}
\dnu (t_{-j} \o_i^{-j}) = \begin{cases}
\frac{-e}{j} t_{-j} \o_i^{-j} & \text{if  $j \ne 0$;} \\
0 &  \text{if  $j =0$.}
\end{cases}
\end{equation*}
Note that if $t = \sum_{i = -N}^\infty t_i \o_i$, then
$\d_e (\dnu (t)) = t - t_0$.
We obtain an induced map
$\dnu : \Om^*_{\tMg} (\ft_i^{-r_i}/ \ft_i^1) \to \Om^*_{\tMg} (\ft_i^{-r_i}/ \ft_i^1)$.

We shall clarify the dictionary between this notation and the notation
in Section~\ref{4.1}. Set $\dAi = \dnu (d\hmA_i ) \in \Om^1_{\tMg}
(\ft_i^{-r_i} / \ft_i^1)$.  If $\nu = \frac{dz_i}{z_i}$, the first
part of Lemma~\ref{del} implies that
\begin{equation}\label{deleqn}
\tau_\nu \dAi  - [\dAi, \hmA_{i \nu}] = d \hmA_{i \nu} - \dto_i+\ft^1_i.
\end{equation}
Below, we will apply Proposition~\ref{totalconnection} several times
with $\dAi$ and  $\hmA_{i \nu}-\hmA_i^0$ playing the roles of $A_\De$ and
$\tA_\nu$ respectively.

\item
  Recall, from \eqref{regprod}, that $(\fP_i / \fP_i^{r_i+1})^\vee_\reg \cong \A_i
  \times_{T^1_i} P_i^1 / P_i^{r+1} $.   Given $m_i \in \tM_i,$ we may write
  $\pi_{\fP_i} (v_i) = (a_i, p_i) \in \A_i \times_{T^1_i}
  P_i^1/P_i^{r_i+1}$.  Thus, $a_i$ is the image of $v_i$ under the map
  $\mA_i$ from Proposition~\ref{A}, and
  $p_i$ is characterized by $\Ad^*(p_i)  (\pi_{\fP_i} (v_i)) = a_i$.  
  Define a bundle
  $B_i = \fP_i^{-r} / \fP_i^1 \times_{U_i} \GL_n(\C)$ over $U_i
  \backslash \GL_n(\C)$,   where
  $U_i$ acts on $\GL_n(\C)$ by left multiplication and on $\fP_i^{-r}
  / \fP_i^1$ by $\Ad$.  We define $\Phi_i \in \Omega^1_{\tMg}
  (\psi_i^* (B_i))$ by
\begin{equation}\label{Phii}
\Phi_i ( m) = \Ad(p_i^{-1}) (\dAi).
\end{equation}

Note that $\Phi_i$ is a well defined section of $\Omega^1_{\tMg}
  (\psi_i^* (B_i))$:  if $u \in U_i$, then $u(a_i,p_i)=(a_i,p_i u^{-1})$;
  it follows that
$
\Ad(u) \Phi_i (u m) = \Phi_i ( m).
$

\item We define a third form $\Up^0$ on $ \tMg$
which has coefficients in $H^0 (\proj ; \sgr)$.
The form
\begin{equation*}
\bPhi_i = \Ad (g_i^{-1}) \Phi_i +\Om^1_{\tMg} (\fg_i /\fP_i^1)
\in \Om^1_{\tMg} ((\fP_i^{-r_i} )^{g_i}/ \fg_i)
\end{equation*}
is well defined on $\tMg$, so the product
$\bPhi = (\bPhi_i)_{i \in I}$ determines a form in $\Om^1_{\tMg} 
(\prod_{i \in I} ((\fP_i^{-r_i})^{g_i}/ \fg_i))$.  Using Definition~\ref{Om0}, 
we set $\Up^0 = \vs (\bPhi) \in \Om^1_{\tMg} (H^0(\proj; \sgr)) $.
%

\item The next form,  $\Th_i$, is defined on
$\tMg$ with coefficients in $\psi_i^*(T (U_i \backslash \GL_n(\C)))$.
Here, we identify $T(U_i \backslash \GL_n(\C))$ with $\fu_i \backslash \gl_n(\C)   
\times_{U_i} \GL_n(\C)$
in the usual way.  
Define a section of $\Omega^1_{\tMg} ( \psi_i^* T(U_i \backslash \GL_n(\C)))$
by
\begin{equation*}
\Th_i (m) =  
(d g_i) g_i^{-1} + \bphi_i(\Phi_i) - \bphi_i (\Ad(g_i) (\Up^0))+
\Om^1_{\tMg}(\fu_i),
\end{equation*}
 where $\bphi$ is the function defined in  \eqref{bphi}.
The form $\Th_i$ is a well defined.  By construction, $\bphi_i (\Phi_i)$
is a well defined section of $\Omega^1_{\tMg} ( \psi_i^* T(U_i \backslash \GL_n(\C)))$.  Furthermore, $\Ad(u_i)\bphi_i (\Ad(g_i)
\Up^0) (u_i m) = \bphi_i (\Ad(g_i) \Up^0) ( m)$ 
and
\begin{equation*}
  d (u_i g_i) g_i^{-1} u_i^{-1}  \in \Ad(u_i) ((dg_i) g_i^{-1}) + \Om^1_{\tMg}(\fu_i).
\end{equation*}

\item  By Theorem~\ref{modspace}, each point of  $\tMg$
determines a unique form $\a \nu \in H^0 (\proj \backslash \bfx;
\Om^1_{\proj} (\sg))$.
In particular, $d \a$ is a one-form with coefficients in $H^0 (\proj \backslash \bfx;\sg)$.
Note that $d \a$ depends on the choice of global form $\nu$.
\item We define the last collection of forms on $\tMg$ to have coefficients 
in $(\fg_i^{-r'_i}\frac{dz_i}{z_i \nu}) / (\fg_i^1\frac{dz_i}{z_i \nu})$, where $\bfre = (r_i')_{i \in I}$ as in Definition~\ref{spgr}.
Let $\Xi_i$ be defined on $ \tMg$ by
\begin{equation*}
\Xi_i((v_j, g_j)_{j \in I}) = \tau_\nu  \Up^0 - d \a - [\Up^0, \a] + 
\Om^1_{ \tMg}( \fg_i^1 \frac{dz_i}{z_i \nu}).
\end{equation*}
\end{enumerate}
Intuitively, we have a framed global deformation on any neighborhood
of $\tMg$ isomorphic to an analytic polydisk determined by $\bfg =
(g_i)_{i \in I}$ and $\tn = d + \a \nu$ (see
Definition~\ref{deformation}).  Here, $\Up^0$ represents the $dy$ term
of the connection (up to a section of $\Om^1_{\tMg} (\gl_n(\C))$).
The form $\Xi_i$ corresponds to the expansion, in non-positive degrees
of $z_i$, of the $dz \wedge dy$ term of the curvature.  Moreover, the
vanishing of $\Th_i$ gives a relation on the framing $g_i$ analogous
to part~\eqref{flow} of Theorem~\ref{defeq}.


\subsection{The Differential Ideal}\label{theideal}
We are now ready to construct a differential ideal $\I$ on $\tMg$
corresponding to the isomonodromy equations in Section~\ref{four}.
Throughout, we fix sections $\hg_i$ of the trivial
$\GL_n(\fo_i)$-bundle over $\tMg$ such that $\hg_i g_i^{-1} \in P_i^1$
and $\hg_i \cdot \a = \hmA_{i\nu}$.  We write $\hp_i = \hg_i
g_i^{-1}$.  Observe that $\Ad(g_i) (\hg_i^{-1}d\hg_i) \in (d g_i)
g_i^{-1} + \Om^1_{\tMg} (\fP_i^1)$ and $p_i = \hp_i
P_i^{r_i+1}$.

For convenience, we will  take $\dAi' \in
\Om^1_{\tMg}(\ft_i^{-r})$ to be the lift of $\dAi$ which has zero components in
positive degrees.  By Lemma~\ref{del}, this lift has the property 
\begin{equation}\label{curvprop}
 \tau_\nu \dAi' + [\hmA_{i \nu}, \dAi'] = d \hmA_{i \nu} - \dto_i.
\end{equation}
Moreover, $ d \dAi' = 0,$
since $\dAi' = \dnu ( d \hmA_{i \nu})$.  The same holds for $\dAi$.

If $B$ is a vector bundle over $\tMg$, $\sigma\in\Om^*_{\tMg}(B)$ and
$\J$ is a differential ideal on $\tMg$, we will
use the shorthand ``$\sigma\in\J$'' to mean
$\sigma\in\I\otimes_{\O_{\tMg}} B$.  If $B'$ is another bundle
and $f:B\to B'$ is a bundle map, then $\sigma\in\J$ implies
that $f(\sigma)\in\J$.  In particular, this is the case if the
bundles are trivial with $f$ induced by a $\C$-linear map on their
fibers.  We also say that $\J$ is generated by $\s$ if it is locally
generated by the coefficients of $\s$ in a trivialization of $B$.

Throughout, we may assume that $\nu = \frac{dz_i}{z_i}$ when working
with $\Xi_i$ and $\Th_i$.  Let $\chi_i = \phi_i (\Ad(\hg_i^{-1}) \dAi'
+ \hg_i^{-1} d \hg_i - \Up^0) \in \Om^1_{\tMg} (\gl_n(\C)).$ Note that
$\Ad(g_i) (\chi_i) + \Om^1_{\tMg} (\fu_i)$ is a well-defined section
of the bundle $B_i$ from the previous section.  Choose some $j \in I$.
Let $\I$ be the differential ideal on $\tMg$ is generated by $\Xi_i -
[\chi_j, \a]$, $\Th_i - \Ad(g_i) (\chi_j)$, and $\dto_i$ for all $i
\in I$.  We will show late that this definition is independent of the
choice of $j$.

\begin{theorem}\label{int}\mbox{}
\begin{enumerate}\item\label{inta}
  The differential ideal $\I$ is a Pfaffian ideal which satisfies the
  Frobenius integrability condition $d \I \subset \I$ and therefore
  determines a foliation of $\tMg(\bfx, \bfP, \bfr)$.  Moreover, any
  map $f : \De \to \tMg$ corresponds to a framed integrable
  deformation if and only if $f^* \I = 0$.
\item\label{intb} The ideal $\I$ is the pullback of an integrable Pfaffian
  differential ideal $\bI$ on $\tM(\bfP,\bfx,\bfr)$, and any
  map $f : \De \to \tM$ corresponds to a framed integrable
  deformation if and only if $f^* \bI = 0$.
\end{enumerate}
\end{theorem}
We will prove this theorem in the subsequent two sections.  The
remainder of this section is devoted to a local analysis of the ideal
$\I$ which will be useful later on.

Over a sufficiently small neighborhood $N_i \subset U_i \backslash
\GL_n(\C)$, we may choose a smooth algebraic slice of coset
representatives $\sigma_i : N_i \to \GL_n(\C) $.  Let $\tN_i \subset
\GL_n(\C)$ be the inverse image of $N_i$, and let $N = \bigcap_{i \in
  I} \psi_i^{-1} (N_i) \subset \tMg$.

For brevity, we will write $g_i^\s = \s_i (U_i g_i)$.  When $\chi \in
\gl_n(\C) = T_{g_i^\s} \GL_n(\C)$, we write $\chi^\s \in \gl_n(\C)$
for the image of $\chi$ under the tangent map corresponding to the
composition $\tN_i \to N_i \xrightarrow{\s_i} \tN_i$.  We also define
$\Th_i^\s \in \Om^1_{N} (\gl_n(\C))$ by 
\begin{equation}\label{Thi}
\Th_i^\s  =  
 (g_i^\s)^{-1}d g_i^\s + T\s_i (\bphi_i(\Phi_i) - \bphi_i (\Ad(g_i) (\Up^0))).
\end{equation}
Observe that $\Ad(g^\s_i) (\Th_i^\s) + \Om^1_{N} (\fu_i) =
\Th_i|_{N}$; indeed, $\Th_i^\s = T \s_i (\Th_i)$.

We define a map $\xi_i : \fg_i \to \Om^1_{N}$ using the residue-trace
pairing in \eqref{pair}:
\begin{equation}\label{xii}
\xi_i (X) = \langle \Ad(g_i^\s)(\Xi_i -[\Th_i^\s, \a]), X \rangle_\nu.
\end{equation}

\begin{lemma}\label{image}
  Let $\J$ be a differential ideal on $N \subset \tMg$, and let
  $\Xi$ be a section of  $\Om^1_{N} (\fP^{-R} / \fg^1)$ for some integer $R> 0$.  
  Define $\xi \in \Om^1_{N}(\fg^\vee)$ by 
  $\xi (X) = \langle \Xi, X \rangle_\nu$.
  Then $\xi  (\fg) \subset \J$ if and only if $\Xi \in \J$.  Furthermore, $\xi (\fP^s) \in \J$
  if and only if $\Xi \in \J + \Om^1_N (\fP^{-s+1})$
\end{lemma}
\begin{proof}
  By assumption, $\fP^{R+1} \subset \ker (\xi)$, so
  $\xi$ induces an element $\bar{\xi} \in \Hom_\C(\fg/ \fP^{R+1},
  \Om^1_{N})\cong \Om^1_N((\fg / \fP^{R+1})^\vee)$.  Note that
  $\xi (\fg) \subset \J$ if and only $\bar{\xi}\in\J$.
  Moreover, $\bar{\xi}$ and $\Xi$ correspond under
  the isomorphism $(\fg / \fP^{R+1})^\vee\cong
  \fP^{-R}/\fg^1$ given by the perfect pairing $\langle,
  \rangle_\nu$, so $\bar{\xi} \in\J$ if and only if
  $\Xi \in\J$.  The first result now follows. A similar argument, replacing
  $\fg$ with $\fP^{s}$,  proves the second statement.
\end{proof}

We will write $\hr$ for the action of $\GL_n(\C)$ on $\prod_{i \in I}
\tM_i$.  (This is the product over the action of $\GL_n(\C)$ on each
factor $\tM(P,r)$ from Lemma~\ref{faction}.)  Define $\TGMg \subset
\prod_{i \in I} \TMi$ to be the subbundle of $\T(\prod_{i \in I}
\tM_i)$ consisting of vectors tangent to the $\GL_n(\C)$-orbits.
Since the action of $\GL_n(\C)$ is free, this is an integrable
subbundle.  Moreover, the action map of $\hr$ gives an isomorphism
between the trivial $\gl_n(\C)$-bundle on $\prod_{i \in I} \tM_i$ and
$\TGMg$.

We may identify the preimage $\tpsi_i^{-1}(N_i)$ of $N_i$ in $\tM_i$ with
$\fV_i \times N_i$ via
\begin{equation*}
(v, g_i) \mapsto (\Ad(g_i^\s g_i^{-1}) v, U_i g_i).
\end{equation*}
Thus, there is  a local isomorphism 
\begin{equation}\label{localiso}
  \T(\tpsi^{-1}_i(N_i))\cong \T(\fV_i)\times \T(N_i). 
\end{equation}
This equation induces a local bundle map $\piU : \TGMg|_N \to
\T(\fV_i) \subset T(\prod_{i \in I} \tpsi^{-1}(N_i))|_N$.  We also
have a global bundle map $\piL : \TGMg|_N\to \psi_i^{-1}T(U_i
\backslash \GL_n(\C))$.


Recall from Lemma~\ref{altdescr} that 
$\fV_i$ is open dense in $Z_i = (\fg_i / \tW_{r_i})^\vee$.
 Therefore, $T(\fV_i) \cong \fV_i \times Z_i$.
\begin{lemma}\label{TGM} 
Let $\chi$ be a section of $\Om^1_N (\TGMg)$.   Then,
$\piL (\chi) = -\Ad(g_i) (\chi) + \fu_i \in \Om^1_N (B_i)$.
Moreover, $\piU (\chi) \in \Om^1_N (Z_i^\vee)$ is the functional defined by
$\Ad(g_i^\s) \left( [\chi-\chi^\s, \a] \right)$.
\end{lemma}
\begin{proof}
  The action of $\GL_n(\C)$ on $U_i \backslash \GL_n(\C)$ is the usual
  (left) action, $h (U_i g) = U_i g h^{-1}$.  The first statement
  follows.  The map $\tN_i \to U_i$ defined by $g \mapsto g
  (g^\s)^{-1}$ has tangent map $X \mapsto \Ad(g^\s) (X - X^\s)$.
  Since the projection from the $\tpsi_i(N_i)\subset\tM_i$ onto
  $\fV_i$ has the form $(v_i, g_i) \mapsto \Ad(g_i^\s g_i^{-1})
  (v_i)$, it follows that $\piU (\chi) = [\Ad(g_i^\s) ( \chi-\chi^\s),
  \Ad(g_i^\s)(\a)])$.
\end{proof}

Choose a vector space lift $\hZ_i \subset \fg_i$ of $\fg_i /
\tW_{r_i}$.  Define $\Xi_i^\fV \in (\hZ_i)^\vee \cong Z_i$ by $
\Xi_i^\fV= \xi_i |_{ \hZ_i}, $ using the notation in \eqref{xii}.  In
other words, $\Xi_i^\fV$ is the functional on $\hZ_i$ determined by
$\Ad(g_i^\s) (\Xi_i - [\Th_i^\s, \a])$.  The $I$-tuple $\tk =
(-\Xi_i^\fV, \Th_i)_{i \in I}$ determines a section of $\Om^1_{N}
(\prod_{i \in I}\TMi)$.  We define $\ka$ to be the differential form
on $\Om^1_{N} ((\prod_{i \in I} \TMi)/ \TGMg)$ determined by the image
of $\tk$.

\begin{definition}
  We define $\I^\s$ to be the differential ideal on $N$ determined by
  the coefficients of $\ka$ and by $\dto_i$ for each $i \in I$.
\end{definition}
We wish to show that $\I^\s$ is the restriction of $\I$ to $N$. 


\begin{lemma}\label{dfqlem}
  Suppose that $\J$ is a differential ideal on $N$ containing
  $\dto_i$ and $\chi$ is a section of $\Om^1_N(\TGMg)$.  If
\begin{equation*}
\Ad(g_i)(\Up^0) + \Ad(g_i)(\chi) -   (d g_i) g_i^{-1}  \in \Phi_i  + \J+ \Om^1_{N} (\fP^1_i) ,
\end{equation*}
then $\Ad(g_i) (\Xi_i -[\chi,\a] )
\in \Om^1_{N} ([\fP^1_i, \Ad(g_i) (\a)] + \fP^1_i\frac{dz_i}{z_i \nu}) +\J.$
\end{lemma}
\begin{proof}
  Without loss of generality, let $ \nu = \frac{dz_i}{z_i}$.  By construction,
  $\tau_\nu (\dAi') - d \hmA_{i \nu} + \dto_i+ [\hmA_{i \nu},
  \dAi'] = 0$.  Write $\a' = \hg_i^{-1} \cdot (\hmA_{i \nu} -
  \hmA_i^0) $, so $\a - \a' = \Ad(\hg_i^{-1}) (\hmA_i^0)$.  Set $\Upsilon =
  \Ad(\hg_i^{-1}) \dAi' + \hg_i^{-1} d \hg_i$.

 By  Proposition \ref{totalconnection}, $\hg_i^{-1} \cdot \left( (\hmA_{i \nu} -
  \hmA_i^0) \nu + \dAi' \right) + \hg_i^{-1} d \hg_i = \a' \nu + \Upsilon$
  determines a flat connection on $\Spec(F_i) \times N$.  Therefore,
 $\tau_\nu (\Upsilon) - 
d \a' + [\a', \Upsilon]  = 0.$
 It follows that
 \begin{equation*}
\tau_\nu(\Upsilon) - d \a + [\a, \Upsilon] = -\Ad(\hg_i^{-1}) \dto_i \in  \J.
 \end{equation*}
 By definition, $\Ad(\hg_i^{-1}) \dAi + \Om^1_{N}
 ((\fP_i^1)^{g_i})= \Ad(g_i^{-1}) \Phi_i$ and $\hg_i^{-1} d\hg_i \in
 g_i^{-1} d g_i + \Om^1_{N}((\fP_i^1)^{g_i})$. Therefore,
 subtracting the above expression from $\Xi_i- [\chi, \a]$ and using the
 hypothesis yields
\begin{multline*}
\Xi_i - [\chi, \a] \in  \tau_\nu (\Up^0 + \chi - \Ad(\hg_i^{-1} )\dAi - \hg_i^{-1} d \hg_i ) -
\\
 [\Up^0 + \chi  - \Ad(\hg_i^{-1} )\dAi - \hg_i^{-1} d \hg_i ,\a] 
+
 \Om^1_{N}( (\fP^1_i)^{g_i}) + \J.
 \end{multline*}    
 The lemma follows after applying $\Ad(g_i)$ to both sides of the equation above.
\end{proof} 

\begin{lemma}\label{modIde1}
Let $\J$ be a differential ideal on $N$ satisfying the following property:
there exists a section $\chi$ of $\Om^1_N(\TGMg)$ such that
$\Th_i +\piL(\chi) \in \J$ for all $i$.  Then,
\begin{equation*}\label{ceqlem}
\Ad(g_i) (\Up^0 + \chi) \in  \Ad(\hp_i^{-1}) \dAi' + (dg_i) g_i^{-1}
+  \Om^1_{\tMg} (\fP_i^1)+ \J.
\end{equation*}
\end{lemma}
\begin{proof}
We note that $\Up^0 -
  \Ad(\hg_i^{-1}) (\dAi') \in \Om^1_{\tMg} (\fg_i)$ by construction.
Since the map $\bphi_i$ induces an isomorphism
$\fP_i^1 \backslash \fg_i \cong \fu_i \backslash \gl_n(\C)$,
it suffices to show 
\begin{equation*}
\bphi_i (\Ad(g_i) (\Up^0 + \chi) - \Ad(\hp_i^{-1}) (\dAi') - (dg_i) g_i^{-1} ) \in \J.
\end{equation*}
This is equal to $-\Th_i + \Ad(g_i) (\chi) + \Om^1_N (\fu_i)$ as a
section of $\Om^1_{N} (\psi_i^* T (U_i \backslash \GL_n(\C)))$.  By
Lemma~\ref{TGM},
\begin{equation*}
-\Th_i + \Ad(g_i) (\chi) + \Om^1_N (\fu_i)= - \Th_i - \piL (\chi)  \in \J.
\end{equation*}
It follows that $\Ad(g_i) (\Up^0 + \chi) \in 
 \Ad(\hp_i^{-1}) (\dAi') + (dg_i) g_i^{-1}  + \Om^1_{\tMg} (\fP_i^1) + \J $.

\end{proof}
\begin{proposition}\label{indc}  On a sufficiently small neighborhood
  $N$, there exists a section $\chi$ of $\Om^1_N(\TGMg)$ such that
  $\tk+\chi\in\I^\s$.  Any such $\chi$ will satisfy the conditions of
  Lemma~\ref{modIde1} for $\J = \I^\s$; furthermore, $\I^\s$ is
  generated by $\Xi_i - [\chi, \a]$, $\Th_i + \piL (\chi)$, and
  $\dto_i$.  In particular, $\I^\s$ is independent of the choice of
  lift $\hZ_i$.


\end{proposition}

For the next few results, we will assume that $N$ is taken to be
sufficiently small as in the statement of the lemma.  However, we
shall see below in Corollary~\ref{chicorollary} that we may choose
$\chi$ to be the restriction of the globally defined $\chi_j$.
\begin{proof}
Shrinking $N$ if necessary, we may choose a direct sum decomposition
$\prod_{i \in I} \TMi)|_{N} \cong \TGMg\vert_{N} \oplus \TGMg^\perp$.
Here,  $\TGMg^\perp \cong ((\prod_{i \in I} \TMi)/ \TGMg)|_{N}$.
Thus, we may choose a section
$\chi$ of $\Om^1_N (\TGMg)$ such that $\tk +\chi \in  \I^\s$.  In particular,
$\Th_i +\piL(\chi) \in \I^\s$ and $-\Xi_i^\fV + \piU (\chi) \in \I^\s$.  
By Lemma~\ref{modIde1}, 
\begin{equation*}
\Ad(g_i) (\Up^0) + \Ad(g_i) (\chi) - (d g_i) g_i^{-1} \in 
\Ad(\hp_i^{-1}) (\dAi') + \I^\s + \Om^1_{\tMg} (\fP_i^1).
\end{equation*}
The right hand side is equal to $\Phi_i + \I^\s + \Om^1_{\tMg}
(\fP_i^1)$.  We may conclude, using Lemma~\ref{dfqlem}, that $\Ad(g_i)
(\Xi_i - [\chi, \a]) \in \Om^1_N(\fP_i^{-r_i+1}) + \I^\s$.

 Consider the identity
\begin{equation}\label{eqn5.7}
\Ad(g_i^\s) (-\Xi_i + [\Th_i^\s, \a]  +[\chi - \chi^\s, \a])=
\Ad(g_i^\s) (-\Xi_i + [\chi, \a]) + \Ad(g_i^\s) ([\Th_i^\s - \chi^\s, \a]).
\end{equation}
Note that $\Th_i^\s - \chi^\s \in \Om^1_N (\gl_n(\C))$ lies in
$\I^\s$, since it is the image of $\Th_i + \piL (\chi)$ under the
tangent map of $\s_i$.  We obtain that the right hand side of
\eqref{eqn5.7} lies in $\Om^1_N(\fP_i^{-r_i+1}) + \I^\s$ since
$\Th_i^\s - \chi^\s \in \Om^1_N (\gl_n(\C))$ is the image of $\Th_i +
\piL (\chi) \in \I^\s $ under the tangent map of $\s_i$.  The left
hand side determines a functional $\xi_i' \in \Om^1_{N} (\fg_i^\vee)$
using the residue-trace pairing. Equation~\eqref{image} and the second
part of Lemma~\ref{TGM} show that $\xi_i'$ is equal to $-\xi_i + \piU
(\chi)$.  It follows that $(-\xi_i + \piU (\chi)) ( \fP_i^{r_i})
\subset \I^\s$.  Since $\tW_{r_i} \subset \fP_i^{r_i}$, we see that
$\hZ_i + \fP_i^{r_i} = \fg_i$.  Thus, $(-\xi_i + \piU (\chi)) (\fg_i)
\subset \I^\s$.  Lemma~\ref{image} implies that both sides of
\eqref{eqn5.7} lie in $\I^\s$.  In particular, $\Xi_i - [\chi,
\a]\in\I^\s$.

Finally, let $\I'$ be the ideal generated by $\Xi_i - [\chi, \a]$,
$\Th_i + \piL (\chi)$, and $\dto_i$.  Since $-\Xi_i^\fV + \piU (\chi)$
is the restriction of the difference of the first two generators, it
lies in the ideal $\I'$.  Therefore, $\kappa' = \tk + \chi$ lies in
$\I'$.  The image of $\kappa'$ in $\Om^1_N ((\prod_i (\TMi)) / \TGMg)$
is $\kappa$, so $\kappa \in \I'$.  It follows that $\I^\s = \I'$.

\end{proof}
We will now show that $\I^\s$ is the restriction of $\I$ to $N$.

\begin{lemma}\label{ideal} Let $\J$ be a differential ideal on $N$.
Suppose that $X \in \Om^1_{\tMg}(\fP^1_i)$ satisfies
\begin{equation}\label{taux}
\tau_\nu (X) - [X, \hmA_{i \nu}] \in \J.
\end{equation}
Then, $X \in \J$.
\end{lemma}
\begin{proof}

Without loss of generality, assume that $\nu = \frac{dz_i}{z_i}$
and $\tau_\nu = z_i \pd_{z_i}$.  To simplify notation, we will
suppress the subscript $i$ in the proof.   Let
$\psi:\Om^1_{\tMg}(\fP^1)\to\Om^1_{\tMg}(\fP^{-r+1})$ be defined by
$\psi(Y)=\tau_\nu (Y) - [Y, \hmA_{\nu}]$.


First, we prove the case $r = 0$, so $e = 1$ and $P =\GL_n(\fo)$.
Note that $\hmA_{ \nu} \in (\tfl)'$ by Definition~\ref{tM} and
$\tau_\nu (Y) = \ell Y$ for $Y\in\Om^1_{\tMg}(z^{\ell}\gl_n(\C))$.  It
follows that $\psi : z^\ell \gl_n(\C) \to z^\ell \gl_n(\C)$ is an
isomorphism, since $-\ell$ is never an eigenvalue of
$\ad(\hmA_{\nu})$.  Now, suppose that $X \in
\Om^1_{\tMg}(\fg^\ell)+\J$ for $\ell \ge 1$, say $X\in
X_\ell+\Om^1_{\tMg}(\fg^{\ell+1})+\J$ for some
$X_\ell\in\Om^1_{\tMg}(z^\ell \gl_n(\C))$.  (This holds trivially for
$\ell=1$.)  We see that the image of
$\psi(X_\ell)\in\Om^1_{\tMg}(\fg^{\ell+1})+\J$.  However, we
have shown that $\psi$ induces an automorphism of
$\Om^1_{\tMg}(\fg^\ell/\fg^{\ell+1})$, so $X_\ell\in
\Om^1_{\tMg}(\fg^{\ell+1})+\J$.  By induction, we conclude
that $X\in\J$.

Let $W_\ell = \ker (\bpit) \subset \fP^\ell / \fP^{\ell+1}$.  We
define the projection $\pt : \gl_n(F) \to \ker(\pit)$ by $\pt (X) = X
- \pit(X)$, and write $\bpt : \fP^\ell/ \fP^{\ell+1} \to W_\ell$ for
the induced map on cosets.  Let
$\psi:\Om^1_{\tMg}(\fP^1)\to\Om^1_{\tMg}(\fP^{-r+1})$ be defined by
$\psi(Y)=\tau_\nu (Y) - [Y, \hmA_{\nu}]$. Since Lemma~\ref{del} states
that $\psi (t) = \d_e (t)\in\ft$ for all $t\in\Om^1_{\tMg}(\ft^1)$,
$\pt\circ\psi(t)=0$.  This implies that $\pt (\psi (Y)) = \pt (\psi
(\pt(Y)))$ for all
$Y\in\Om^1_{\tMg}(\fP^\ell)$. 

First, we show inductively that $X'=\pt(X)\in\J$.  Suppose
that $X' \in \J + \Om^1_{\tMg}(\fP^\ell)$, say $X'\in X'_\ell
+ \J$ with $X'_\ell\in\Om^1_{\tMg}(\fP^\ell)$.  Since
$\psi(X)\in\J$, we obtain $\pt (\psi (X'_\ell))=\pt (\psi
(X))+\J\subset \J$.  By part~\eqref{cores5} of
Proposition~\ref{cores}, the map $\bpsi : W_\ell \to W_{\ell - r}$
induced by $\pt \circ \psi$ is an isomorphism, so $X'_\ell \in
\J + \Om^1_{\tMg} (\fP^{\ell+1})$. It follows that $X' \in
\J + \Om^1_{\tMg} (\fP^{\ell+1})$.  Since the inductive
hypothesis holds trivially for $\ell=1$, we conclude that
$X'\in\J$.

 We now have $X\in\pit(X)+\J$ and also
 $\d_e(\pit(X))=\psi(\pit(X))=\psi(X)-\psi(X')\in\J$.  Since
 $\d_e$ is an automorphism of $\ft^1$, this gives
 $\pit(X)\in\J$, and hence $X \in \J$.

\end{proof}

In the following, 
let $\chi$ be the section of $\TGMg|_N$
constructed in Lemma~\ref{TGM}.
\begin{proposition}\label{modI}
The following identity holds modulo $\I^\s$:
\begin{equation*}\label{ceq2}
\Up^0 +  \chi \in  \Ad(\hg_i^{-1}) \dAi' + \hg_i^{-1} d \hg_i 
+ \I^\s.
\end{equation*}
\end{proposition}

\begin{proof}

By Proposition~\ref{indc},
\begin{equation*}
\tau_\nu \Up^0 - d \a - [\Up^0+ \chi, \a] + \Om^1_{\tMg}(\fg_j^1 \frac{dz_j}{z_j \nu} ) = 
 \Xi_j  - [\chi, \a] \in \I^\s.
\end{equation*}
Thus, if $\nu$ is a global meromorphic one-form on $\proj$, we obtain
$\nu \wedge \left(\tau_\nu \Up^0 - d \a- [\Up^0 + \chi , \a] \right)
\in \I^\s$.  This follows from the observation that the principal part
at each $x_j \in \proj$ lies in $\I^\s$.  Applying the interior
derivative $\io_{\tau_\nu}$ implies that
\begin{equation}\label{gloform} \tau_\nu \Up^0 - d \a- [\Up^0 + \chi ,
  \a] \in \I^\s.
\end{equation}

Now, we set $\nu = \frac{dz_i}{z_i}$ and $\tau_\nu = z_i \pd_{z_i}$.
Note that $\Up^0 - \Ad(\hg_i^{-1}) (\dAi') \in \Om^1_{\tMg} (\fg_i)$
by construction.  Therefore, setting $\hp_i^\s = \hg_i (g_i^\s)^{-1}$,
\begin{multline*}
\Ad(g_i^\s) (\Up^0) + \Ad(g_i^\s)(\chi) - (\Ad((\hp^\s_i)^{-1}) (\dAi') + (dg^\s_i) (g^\s_i)^{-1} )  \in \\
-\phi_i ((d g^\s_i) (g^\s_i)^{-1} + \Ad((\hp^\s_i)^{-1}) (\dAi') 
-\Ad(g^\s_i) (\Up^0)) + \Ad(g^\s_i)( \chi) + \Om^1_{\tMg}(\fP^1_i) =
\\  -\Ad(g_i^\s)(\Th_i^\s -\chi) \in \I^\s.
\end{multline*}
It follows that 
$\Up^0 +  \chi \in  \Ad(\hg_i^{-1}) \dAi' + \hg_i^{-1} d \hg_i 
+ \I^\s + \Om^1_{\tMg} (\fP^1_i).$


Let $\Up' = \Ad(\hg_i) (\Up^0 +\chi) - (d \hg_i) \hg_i^{-1}$.
From the argument above, we see that $\Up' - \dAi' \in
\Om^1_{\tMg}(\fP_i^1) + \I^\s$.  Substituting $\Ad(\hg_i^{-1})
(\Up') + \hg_i^{-1} d \hg_i = \Up^0 + \chi$ and $\Ad(\hg_i^{-1})
\hmA_{i\nu} + \hg_i^{-1} \tau_\nu (\hg_i) = \a$ into \eqref{gloform},
we obtain
\begin{equation*}
\tau_\nu (\Up') - d \hmA_{i \nu} - [\Up', \hmA_{i \nu}] \in \I^\s.
\end{equation*}

Since
$ d \hmA_{i \nu} - \dto_i = \tau_\nu (\dAi') - [\dAi', \hmA_{i \nu}]$ by
\eqref{curvprop},
we deduce that
\begin{equation*}
\tau_\nu (\Up'- \dAi') -
[\Up'- \dAi', \hmA_{i\nu}] \in \dto_i +  \I^\s.
\end{equation*}
The left hand side lies in $\I^\s$, since $\dto_i \in \I^\s$.
The statement of the Lemma now follows, since by Lemma~\ref{ideal}, $\Up' -
\dAi' \in \I^\s$.





\end{proof}


\begin{corollary}\label{chicorollary}
  The sections $\chi_i$ of $\TGMg$ satisfy $\tk+\chi_i\in\I^\s$.
  Moreover, $\chi_i \in \chi_j + \I^\s$ and $d \chi_i + \chi_i \wedge
  \chi_i \in \I^\s$.
\end{corollary}
\begin{proof}
  Let $\chi$ be the section of $\TGMg$ constructed in
  Proposition~\ref{indc}.  To prove the first statement, it suffices
  to show that $\chi_i - \chi \in \I^\s$.  However, for any $i\in I$,
\begin{equation*}
\chi_i - \chi = \phi_i (\Ad(\hg_i^{-1}) \dAi' + \hg_i^{-1} d \hg_i - \Up^0 - \chi) \in \I^\s
\end{equation*}
by Proposition~\ref{modI}.  

In order to prove the second part, we recall that $\Up^0$ is defined
as $\vs((\bPhi_j)_{j\in I})$.  It follows that $d \Up^0 = \vs ((d
\bPhi_j)_{j \in I})$.  By construction, $\bPhi_j = \Ad(\hg_j^{-1})
(\dAj)$.  Therefore, $d \bPhi_j \in -[\Ad(\hg_j^{-1}) \dAj, \hg_j^{-1}
d \hg_j] + \Om^1_N(\fg_j)$ since $d \dAi = 0$.
 
By Proposition~\ref{modI}, $\Ad(\hg_j^{-1}) \dAj + \hg_j^{-1} d \hg_j \in \Up^0 + \chi_i + \I^\s$.
We conclude that 
\begin{equation*}
[\Ad(\hg_j^{-1}) \dAj, \hg_j^{-1} d \hg_j] \in -(\Up^0+\chi_i) \wedge 
(\Up^0+\chi_i) + \I^\s + \Om_N^2 (\fg_j),
\end{equation*}
and moreover $ \vs( (d \bPhi_j)_{j \in I}) \in -\vs ( \left((\Up^0 +
  \chi_i) \wedge (\Up^0 + \chi_i)\right)_{j \in I} ) + \I^\s.  $ Thus,
by Lemma~\ref{res0}, $ (\Up^0 + \chi_i) \wedge (\Up^0 + \chi_i) + \vs(
(d \bPhi_j)_{j \in I}) \in \chi_i \wedge \chi_I + \I^\s.  $

Finally, 
\begin{multline*}
d (\Ad(\hg_i^{-1}) \dAi  + \hg_i^{-1} d \hg_i ) =\\
 -(\Ad(\hg_i^{-1}) \dAi  + \hg_i^{-1} d \hg_i )
\wedge (\Ad(\hg_i^{-1}) \dAi  + \hg_i^{-1} d \hg_i)\\
\in -(\Up^0 + \chi_i) \wedge (\Up^0+\chi_i) + \I^\s.
\end{multline*}

It follows that 
\begin{equation*}
\begin{aligned}
d \chi_i + \chi_i \wedge \chi_i &= \phi_i \left(d \left(\Ad(\hg_i^{-1}) \dAi' + \hg_i^{-1} d \hg_i - \Up^0\right)
\right)  
+ \chi_i \wedge \chi_i \\
&  \in \phi_i (- (\Up^0 + \chi_i) \wedge (\Up^0 + \chi_i) + (\Up^0 + \chi_i)\wedge (\Up^0 + \chi_i) - \chi_i \wedge \chi_i) + \chi_i \wedge \chi_i + \I^\s \\
& = \I^\s.
\end{aligned}
\end{equation*}
\end{proof}

\begin{corollary}\label{gloideal}
Fix $j \in I$.
The ideal $\I^\s$ is  generated by $\Xi_i - [\chi_j, \a]$, $\Th_i + \piL(\chi_j)$,
and $\dto_i$ for all $i \in I$.
In particular, $\I^\s$ is the restriction of $\I$ to $N$, and $\I$
does not depend on the choice of $j$.
\end{corollary}
\begin{proof}
By Corollary~\ref{chicorollary}, $\chi_j$ satisfies the the conditions of Proposition~\ref{indc}.
The same proposition implies that $\I^\s$ has the given generators.

\end{proof}

\subsection{Invariance}

In this section, we will show that $\I$ is the pullback of a
differential ideal $\bI$ on $\tM(\bfP, \bfr, \bfx)$.  Moreover, the
ideal $\bI$ is Pfaffian and integrable.

\begin{proposition}\label{propiox} Let $X$ be a section of $\TGMg$,
  and take $\om \in \I$.  Then, $\iota_X \om \in \I$.  In particular,
  if $\om \in \I \cap \Om^1_{\tMg}$, then $\iota_X \om = 0$.
\end{proposition}
\begin{proof}
It suffices to show that the generators of $\I$ are in the kernel of $\iota_X$.
Since $\hmA_i$ is invariant
under the action of $\GL_n(\C)$, we see that $\io_X \dto_i = 0$ and
$\io_X (\dAi) = 0$.  It follows that any term involving $\dAi$ lies  in the kernel
of $\io_X$.
We now calculate $\io_X (\Xi_i - [\chi_j, \a])$.  
By the argument above,
\begin{equation*}
\io_X (\Xi_i ) = \io_X (-d \a) = -[X, \a].
\end{equation*}
On the other hand, 
\begin{equation}\label{iox}
\io_X (\chi_j) = \io_X \phi_j (\hg_j^{-1} d \hg_j)
= - X.
\end{equation}
It follows that $\io_X (\Xi_i - [\chi_j, \a]) = -[X, \a] + [X, \a] = 0.$
On the other hand, 
\begin{equation*}
  \io_X (\Th_i) = \io_X((dg_i) g_i^{-1} + \Om^1_{\tMg} (\fu_i)) = \Ad(g_i) (X) + \fu_i.
\end{equation*}
Using \eqref{iox}, we see that 
$\io_X (\Th_i + \piL(X)) = \Ad(g_i) (X) - \Ad(g_i) (X) + \fu_i = \fu_i$.
\end{proof}

We now show that the ideal $\I$ is $\GL_n(\C)$-invariant.
\begin{proposition}\label{pb}  Let $g:V\to\GL_n(\C)$ be a function defined
  on a neighborhood $V\subset\tMg$, and let $\hr_g:V\to\tMg$ be the
  map determined by the action of $g$.  Then, $\hr_g^* \I = \I$.
  Moreover, $\hr_g^* (\Xi_i - [\chi_j, \a]) = \Ad(g) (\Xi_i - [\chi_j,
  \a])$ and $\hr_g^* (\Th_i+ \piL (\chi_j)) = \Th_i + \piL(\chi_j))$.
\end{proposition}
\begin{proof}
First, we calculate $\hr_g^* (\chi_j)$.  One easily checks that
\begin{equation}\label{hrchi}
\hr_g^* (\chi_j) = \Ad(g) \chi_j - (dg) g^{-1}.
\end{equation}
We also see that $\hr_g^* (\Xi_i) = \Ad(g) (\Xi_i) -[ (dg) g^{-1},
\a]$ and $\hr_g^* (\Th_i) = -\Ad(g_i g^{-1}) ((dg) g^{-1}) + \Th_i$.
We see that $\hr_g^* (\Xi_i - [\chi_j, \a]) = \Ad(g) (\Xi_i - [\chi_j,
\a]) \in \I$ and $\hr_g^* (\Th_i+ \piL(\chi_j)) = \Th_i + \piL (\chi_j)
\in \I$.  Finally, $\hr_g^* \dto_i = \dto_i$.  It follows that
$\hr_g^* (\I) \subset \I$.
\end{proof}

In the following, $\tM = \tM (\bfP, \bfx, \bfr)$ and  $q: \tMg \to \tM$
is the natural projection.
\begin{corollary}
There exists a differential ideal $\bI$ on $\tM (\bfP, \bfx, \bfr)$
such that $q^* \bI = \I$.
\end{corollary}
\begin{proof}

Since $\tMg$ is a principal $\GL_n(\C)$-bundle over $\tM$,
we may choose a cover of $\tM$ by neighborhoods $W$
with the property that $q^{-1} (W) \cong \GL_n(\C) \times W$.  
Define $\Th_i^G$ and $\Xi_i^G$ on $q^{-1} (W)$ by
$\Th_i^G (g, w) = \Th_i + \piL(\chi_j) + \Om^1_{\tMg}(\fu_i)$ and
$\Xi_i^G (g, w) = \Ad(g^{-1}) (\Xi_i - [\chi_j, \a])$.  It is clear that the collection
 $\Th_i^G,$ $\Xi_i^G$, and $\dto_i$ over all $i \in I$ generates $\I|_{q^{-1} (W)}$.
By Proposition~\ref{pb},
$\hr_{h}^* (\Th_i^G) = \Th_i^G$ and $\hr_{h}^* (\Xi_i^G) = \Xi_i^G$.

Proposition~\ref{propiox} shows that each of the generators
determines a well-defined form on $q^* T(W) \cong T(q^{-1}(W)) / \gl_n(\C)$.
By invariance, $\Th_i^G,$ $\Xi_i^G$, and $\dto_i$
descend to differential forms on $T(\tM)$.  These generate the
ideal $\bI$.

\end{proof}

\section{Integrability} \label{descent}

We will now show that $\I$ is an integrable Pfaffian system
in the sense of \cite[II.2.4]{HH}.   In general, if
$\J$ is a differential ideal on a smooth variety $X$, 
we write $\J^1 = \J \cap \Om^1_{X}$.
It suffices to show that the coherent sheaf
$\I^1$ is locally free and that $d \I \subset \I$.  

Let $\J$ be the ideal on $\tMg$ generated by $\Xi_i$, $\dto_i$, and
$\Th_i$ for all $i \in I$.  Note that $\I\subset\J$.  We will first
show that $\J$ is a Pfaffian system and that $\I ^1 \subset \J^1$
has constant corank.


We need to describe a (local) minimal basis for $\J$.  Since $\J$ is
independent of the choice of one-form $\nu$, we may assume that $\nu =
\frac{dz_i}{z_i}$ when working locally at $x_i$.  As in
Section~\ref{theideal}, we will restrict to sufficiently small neighborhoods $N_i
\subset U_i \backslash \GL_n(\C)$ and the corresponding neighborhood
$N\subset \tMg$.


By part \eqref{cores2} of Proposition~\ref{cores}, the map $\piti$
induces a map $\pi'_{\ft_i} : \fP_i/\fP_i^{r_i} \to \ft_i/ \ft_i^{r_i}$.
Choose a  vector space lift $Z^{od}_i \subset \fP_i$ of $\ker
(\pi'_{\ft_i})$ and a lift
$Z^{\fu}_i\subset \fg_i$ of $\fg_i /
\fP_i$.  Note that these vector spaces are trivial when $r=0$.
We define locally the `off-diagonal' and $\fu$ components of
$\Xi_i$ to be $\Xi_i^{od} = 
\xi_i  |_{Z^{od}}\in \Om^1_{N}
((Z^{od})^\vee)$ and $\Xi_i^{\fu} = 
\xi_i  |_{Z_i^{\fu}}\in \Om^1_{N}
((Z_i^{\fu})^\vee)$ respectively. 

\begin{proposition}\label{gens}
  The ideal $\J|_N$ is generated by $\Th_i $, $\Xi_i^{od}$,
  $\Xi_i^{\fu}$, and $\dto_i$ for $i\in I$.
\end{proposition}
\begin{proof}
As above, we assume that $\nu = \frac{dz_i}{z_i}$.
By Lemma~\ref{image}, it suffices to show that $\xi_i(\fg_i)$
lies in the ideal $\J'$ generated by
$\dto_i$, $\Th_i $, $\Xi_i^{od}$, and $\Xi_i^{\fu}$, 
since $\J' \subset \J$.

Let $\J_1$ be the ideal generated by $\Th_i $ and
$\dto_i$ for each $i \in I$.  By Lemma~\ref{modIde1} (applied with $\chi=0$), 
$\Up^0 
\in \Ad(\hg_i^{-1}) \dAi' + \hg_i^{-1} d \hg_i + \Om^1_{W}
((\fP_i^1)^{g_i}) + \J_1.$ Lemma~\ref{dfqlem} states that
\begin{equation}
  \Ad(g_i^\s) \Xi_i  \in \Om^1_{W} ([\fP_i^1, \Ad(g_i^\s) (\a)] +
  \fP_i^1) + \J_1\subset  \Om^1_{W} (\fP^{-r+1}) + \J_1.
\end{equation}
Therefore, $\xi_i (\fP_i^r) \subset \J_{1}$.

We will now show that $\Ad(g_i^\s) (\Xi_i) \in \J' + \Om^1_{N}
(\fP_i^1)$.  Take $1\le j\le r$, and
assume inductively that there exists $X \in \Om^1_{N}(\fP_i^j)$
such that
\begin{equation}\label{indstep}
\Ad (g_i^\s) \Xi_i \in [X, \Ad(g_i^\s) (\a)] + \J'+ \Om^1_{N}(\fP_i^1).
\end{equation}
Part \eqref{cores3} of Proposition~\ref{cores} shows that $[X,
\Ad(g_i^\s) (\a)] + \Om^1_{N} (\fP_i^{j-r+1}) \in \ker (\bpiti)$.
Thus, by part~\eqref{cores4} of the same proposition, $\langle [X,
\Ad(g_i^\s) (\a)], \ft_i^{r-j} \rangle_\nu = \{0\}$.  Since
$\Xi_i^{od}\in\J'$ , $\langle [X, \Ad(g_i^\s) (\a)], Y \rangle_\nu \in
\J'$ for $Y \in \ker (\piti) \cap \fP^{r-j}$.  Combining these two
facts gives $ \langle [X, \Ad(g_i^\s) (\a)], \fP^{r-j} \rangle_\nu \in
\J'$, and we conclude that $[X, \Ad(g_i^\s) (\a)] \in
\Om^1_{N}(\fP_i^{j-r+1}) + \J'$.  Part \eqref{cores5} of
Proposition~\ref{cores} now implies that $X \in \pi_\ft (X) +
\Om^1_{N}(\fP_i^{j+1})+ \J'$.  Finally, since there exists $p \in
P_i^1$ such that $\Ad(p) (\Ad(g_i^\s)(\a)) \in \ft_i^{-r} + \fP_i^1$,
we see that $X - \Ad(p^{-1}) (\pi_\ft (X)) \in \Om^1_{N} (\fP_i^{j+1})
+ \J'$ satisfies \eqref{indstep} for $j+1$.  By induction, we obtain
\eqref{indstep} for $r+1$.  This gives $\Ad(g_i^\s) (\Xi_i) \in \J' +
\Om^1_{N} (\fP_i^1)$ and hence $ \xi_i (\fP_i) \subset \J'$.  Finally,
since $\fg_i = Z_i^\fu + \fP_i$, we see that $\xi_i (\fg_i) \subset
\I'$ as desired.

\end{proof}

In the following, let $\iota : \tMg \to \prod_{i \in I} \tM_i$ be the
inclusion and $T^*\iota : \iota^* (\Om^1_{ \prod_{i \in I} \tM_i}) \to
\Om^1_{\tMg}$ the induced bundle map.  The generators of $\I$ (resp.
$\J$) lift to a set of generators for a differential ideal $\tI
\subset \iota^* \Om^*_{\prod_{i \in I} \tM_i}$ (resp. $\tJ$) on
$\tMg$.  Only the lift $\tXi_i$ of $\Xi_i$ requires explanation.
Suppose that $m = ( (U_j g_j,\a_j)_{ j \in I}) \in \tMg$.  We write
$\a_{i \nu}$ for a representative of $\a$ in $(\fg_i^{-r_i'}
\frac{dz_i}{z_i \nu}) / (\fg_i^1 \frac{dz_i}{z_i \nu} )$.  Therefore,
if $m \in \tMg$, $\a_{i \nu} = \a + \fg_i^1 \frac{dz_i}{z_i \nu}$.  We
set
\begin{equation*}
\tXi_i = \tau_\nu \Up^0 - d \a_{i \nu} - [\Up^0 , \a_{i \nu}].
\end{equation*}
It is clear that $T^*\iota (\tXi_i)=\Xi_i$.  All other forms involved
in the definitions of $\I$ and $\J$ are restrictions of forms on
${\prod_{i \in I} \tM_i}$, and we will not make any notational
distinction between such a form and its restriction to $\tMg$.

\begin{lemma}\label{dmu}
Let $\mu$ be the moment map for the action of
$\GL_n(\C)$ on $\prod_{i \in I} \tM(P_i, r_i)$.  Then, 
$d \mu \subset \tI$.
\end{lemma}
\begin{proof}
  Let $\nu$ be the global choice of one form.  By construction, $\mu (
  (U_i g_i, \a_i)_{i \in I}) = \sum_{i \in I} \res (\a_i)$.  Let $\chi
  = \chi_j$, so $\tXi_i - [\chi, \a]$ is one of the generators of
  $\tI$.  We have
\begin{equation*}
\begin{aligned}
\sum_{i \in I} \Res_i (\nu \wedge (\tXi_i- [\chi, \a])  &= \sum_{i \in I} \Res_i(\nu \wedge \tau_\nu ( \Up^0)) - 
\Res_i (\nu \wedge d \a_{i\nu} ) - \Res_i (\nu \wedge [\Up^0+\chi , \a ]) \\
& = \sum_{i \in I} \Res_i (\nu \wedge d \a_{i \nu}).
\end{aligned}
\end{equation*}
Here, $\sum_{i \in I} \Res_i (\nu \wedge \tau_\nu (\Up^0))$ and 
$\sum_{i \in I} \Res_i (\nu \wedge [\Up^0, \a])$
vanish by the residue theorem.
Since $\Res_i (\nu \wedge \a_{i \nu} ) = \res_i (\a_i)$, we see that
$d \mu = \sum_{i \in I} \Res_i (\nu \wedge d \a_{i \nu})$ lies in $\I$.
\end{proof}


Let $\tMg(\bfA)=q^{-1}(\tM(\bfA))$ and $N(\bfA)=N\cap\tMg(\bfA)$.  We
also set $\bN=\prod_{i \in I} (\tpsi_i)^{-1} (N_i)\subset\prod_{i \in
  I} \tM_i$ and $\bN(\bfA) = \bN \cap \left(\prod_{i \in I}
  \tM(A_i)\right)$.

\begin{proposition}\label{Pfaffian}
  The ideal $\J$ is a Pfaffian system.  Moreover, the restriction map taking
  $\J^1 \subset \Om^1_{\tMg}$ to $\Om^1_{\tMg(\bfA)}$ is surjective.
\end{proposition}
\begin{proof}
  This is a local statement, so we will work with the restriction of
  $\J$ to $N$.  Fix a point $m = (U_i g_i,\a_i)_{i
    \in I} \in N(\bfA)$.  We will first show that the restriction of
  $\tJ^1$ to $\bN(\bfA)$ spans $\Om^1_{\bN(\bfA)}$.

  There are isomorphisms $\bN \cong \prod_i (\fP_i /
  \fP_{r+1})^\vee_{reg}\times N_i$
  and $\bN(\bfA) \cong \prod_{i \in I} (\pPo)^{-1} (\Oo_i)\times
  N_i$.
    Restricting $\tJ$ to $\bN(\bfA)$, we see that all terms 
   in $\Th_i$ and $\tXi_i$ involving
  $\dAi$ vanish.  
   In particular, $\Up^0$ and $\Phi_i$ vanish, so $\Th_i $ becomes
  $(dg_i) g_i^{-1} + \Om^1_{\bN(\bfA)} (\fu_i)$, and $\tXi_i$ is simply
  $d\a + \Om^1_{\bN(\bfA)} (\fg_i^1 \frac{dz_i}{z_i \nu})$.  Write
  $\Ad(g_i) (\a) = v \in (\pPo)^{-1} (\Oo_i)$.  Since $d \a =
  \Ad(g_i^{-1}) (d v -[(dg_i) g_i^{-1}, v])$, it is easily checked
  that the coefficients of $(dg_i) g_i^{-1}$ and $d \a$ span $T^*_{m}
  (\bN(\bfA))$.  
  
  On the other hand, a calculation using \cite[Lemma
  3.17]{BrSa} shows that $\dim T^*_{m} (\bN (\bfA))=\dim \gl_n(\C)/\fu
  +\dim \fg/\fP^r - \dim \ft^1/\ft^r$, and so the coefficients of the
  set of generators for $\tJ$ in Proposition~\ref{gens} give a basis
  for $T^*_{m} (\bN (\bfA))$.  Thus, $\tJ$ is a rank $\sum_{i \in I} \dim
  (\tM(A_i))$ Pfaffian system.

  Lemma~\ref{dmu} shows that $d \mu \in \tJ$.  Since the coefficients
  of $d \mu$ span the conormal bundle of $\tMg$ in $\prod_{i \in I}
  \tM(P_i, r_i)$, we see that the pullback of $\tJ$ in $\Om^*_{\tMg}$,
  i.e., $\J$, is a rank $\left(\sum_{i \in I} \dim (\tM(A_i))\right) - n^2$
  Pfaffian system.  Moreover, the restriction map taking 
  $\J^1$ to $\Om^1_{N(\bfA)}$ induced by the map $\tJ^1 \to \Om^1_{\bN(\bfA)}$
  is still surjective.
    \end{proof}

\begin{lemma}\label{JandI}
Fix $j \in I$.
The ideal $\J$ is generated by $\I$ and the coefficients of $\chi_j$.
\end{lemma}
\begin{proof}
We know that $\I \subset \J$.  By Corollary~\ref{chicorollary}, 
$\tk + \chi_j \in \I$.  However, by construction, $\tk \in \J$; it follows that
$\chi_j \in \J$.  

In order to prove the reverse inclusion, Proposition~\ref{indc} shows
that $\Xi_i - [\chi_j, \a] \in \I$.  By assumption, $\Th_i +
\piL(\chi_j) \in \I$ as well.  It follows that $\Xi_i$ and $\Th_i$ lie
in the ideal generated by $\I$ and $\chi_j$.
\end{proof}

There is a natural restriction map $\Om^1_{\tMg} \to (\TGMg)^\vee$.  
We denote by $F$ the induced map $F : \J^1 \to (\TGMg)^\vee$.
\begin{proposition}
The map $F$ is surjective, and $\I^1$ is the kernel of $F$.
Moreover, $\I$ is a Pfaffian system.
\end{proposition}

\begin{proof}
  First, we prove surjectivity.  For a point $m\in\tMg(\bfA)$, the map
  $F$ factors through restriction to $T((\tMg(\bfA))$, since
  $\tMg(\bfA)$ is $\GL_n(\C)$-invariant.  By
  Proposition~\ref{Pfaffian}, the restriction map $\J^1 \to
  \Om^1_{\tMg(\bfA)}$ is a surjective bundle map.  Since the map
  $\Om^1_{\tMg(\bfA)} \to \TGMg^\vee$ is surjective, it follows that
  $F$ is surjective.

  By Proposition~\ref{propiox}, $\I^1 \subset \TGMg^\perp$.  We
  conclude that $\I \subset \ker(F)$.  Finally, since $\I^1$ and
  $\chi_j$ generate $\J^1$, the rank of $\I^1$ must be at least
  $\rk(\J^1) - n^2$.  The rank of the kernel of $F$ is $\rk(\J^1) -
  n^2,$ so $\I^1 = \ker(F)$.  Therefore, $\I$ has constant rank and is
  thus a Pfaffian system.
\end{proof}
We will now prove that $\I$ satisfies the Frobenius integrability condition.

\begin{lemma}\label{modI2} Set $\chi = \chi_j$ as above.
  Then,
\begin{equation*}
d (\Up^0 + \chi) + (\Up^0 + \chi) \wedge (\Up^0+ \chi) \in \I.
\end{equation*}
\end{lemma}
\begin{proof}


First, we will show that  
\begin{equation}\label{weq}
d (\Up^0) \in -(\Up^0 + \chi) \wedge (\Up^0 + \chi) + \Om^2_{\tMg} (\fg_i) + \I.
\end{equation}
By construction,  $\Up^0 \in \Ad(\hg_i^{-1} (\dAi')) +\Om^1_{\tMg} (\fg_i)$.  Therefore,
\begin{equation*}
d (\Up^0) \in - [\Ad (\hg_i^{-1}) (\dAi'),\hg_i^{-1} d \hg_i] + \Om^2_{\tMg} (\fg_i).
\end{equation*}
On the other hand, by Proposition~\ref{modI}, 
\begin{multline}
(\Up^0 + \chi) \wedge (\Up^0 + \chi)  \in \\
( \Ad(\hg_i^{-1}) \dAi' + \hg_i^{-1} d \hg_i ) \wedge  (\Ad(\hg_i^{-1}) \dAi' + \hg_i^{-1} d \hg_i ) 
+ \Om^2_{\tMg}(\fg_i) + \I.
\end{multline}
The right hand side is equivalent to
$ [\Ad(\hg_i^{-1}) \dAi', \hg_i^{-1} d \hg_i] + \Om^2_{\tMg} (\fg_i)+ \I.$
This proves \eqref{weq}.  

It suffices to show that
\begin{equation}\label{weq2}
  (\Up^0 + \chi) \wedge (\Up^0 + \chi)=
  \vs( ((\Up^0 + \chi) \wedge (\Up^0 +\chi) + \Om^2_{\tMg}(\fg_i))_{i \in I}) + \chi \wedge \chi,
\end{equation}
since the right hand side is equal to  
$-d (\Up^0 + \chi) \pmod{\I}$ by the arguments above and Corollary~\ref{chicorollary}.
By Lemma~\ref{res0}, 
$(\Up^0 + \chi)|_{z_0 = 0}= \chi$.  
Moreover, 
\begin{equation*}
\left((\Up^0 + \chi) \wedge (\Up^0 + \chi)\right)|_{z_0 = 0} = 
\chi \wedge \chi.
\end{equation*}
Another application of  Lemma~\ref{res0} gives \eqref{weq2}.

\end{proof}

\begin{lemma}\label{modI3}
\begin{equation*}
d (\Th_i - \Ad(g_i) (\chi)) \in \Om^2_{\tMg}(\fu_i)+ \I.
\end{equation*}
\end{lemma}
\begin{proof}
  Throughout, we will write $\Up = \Up^0 +\chi$.  As in the beginning
  of Section~\ref{theideal}, we write $\hp_i = \hg_i g_i^{-1}$.  By
  Proposition~\ref{modI}, we see that
\begin{equation*}
\Ad(g_i) (\Up) \in \Ad(\hp_i^{-1}) (\dAi') + \hp_i^{-1} d\hp_i + (dg_i) g_i^{-1} + \I.
\end{equation*}
A sequence of calculations using the expression above and Lemma~\ref{modI2} 
produces
\begin{equation*}\label{pf2}
\begin{aligned}
d (\Ad(g_i)  \Up) & = [(dg_i) g_i^{-1}, \Ad(g_i) (\Up)]
+ \Ad(g_i) d (\Up) \\
& \in  [(dg_i) g_i^{-1}, \Ad(g_i) (\Up)] - \Ad(g_i) \left(\Up \wedge \Up \right) + \I\\
& =  [(d g_i) g_i^{-1}, \Ad(\hp_i^{-1}) (\dAi') + \hp_i^{-1} d \hp_i + (dg_i) g_i^{-1}] 
\\ 
&- (\Ad(\hp_i^{-1}) (\dAi') + \hp_i^{-1} d \hp_i + (dg_i) g_i^{-1}) 
\wedge 
(\Ad(\hp_i^{-1}) (\dAi') + \hp_i^{-1} d \hp_i + (dg_i) g_i^{-1}) + \I \\
& = -dg_i d g_i^{-1} + d \hp_i^{-1} d \hp_i - 
[\hp_i^{-1} d \hp_i, \Ad(\hp_i^{-1}) \dAi')] + \I \\
& = -dg_i dg_i^{-1} + d\hp_i^{-1} d \hp_i + d (\Ad(\hp_i^{-1}) \dAi')) + \I.
\end{aligned}
\end{equation*}
The last line is equivalent to $ d \Phi_i  - dg_i dg_i^{-1} + 
\I \pmod{\Om^1_{\tMg} (\fP_i^1)}$.
Therefore,
\begin{equation*}
  d (\Phi_i + (dg_i) g_i^{-1} - \Ad(g_i) (\Up^0+ \chi) \in \Om^2_{\tMg} (\fP_i^1) + \I.
\end{equation*}
Applying $\bphi_i$ to the equation above gives
$d(\Th_i - \Ad(g_i) (\chi)) \in \Om^2_{\tMg} (\fu_i) + \I.$

\end{proof}

\begin{proof}[Proof of Theorem~\ref{int}, part~\eqref{inta}]
  As above, write $\Up = \Up^0 + \chi$.  We have already shown
  that $d (\Th_i - \Ad(g_i) (\chi)) \in \Om^2_{\tMg} (\fu_i) + \I$.  A standard,
  but tedious, calculation
  shows that
\begin{equation*}
d (\Xi_i-[\chi,\a]) = 
\tau_\nu ( d \Up +\Up \wedge \Up) - [ d \Up + \Up \wedge \Up , \a] +[ \Xi_i, \Up]
+ \Om^1_{\tMg}(\fg_i^1 \frac{dz_i}{z_i \nu}),
\end{equation*}
so by Lemma~\ref{modI2}, $d (\Xi_i-[\chi,\a]) \in \I$.

Finally, by Theorem~\ref{modspace}, we may identify trivialized framed
global deformations $(\bfg, \bV, \tn)$ on $\proj \times \De$ with
analytic maps $\s : \De \to \tMg(\bfx, \bfP, \bfr)$. We will show that
the conditions of Corollary~\ref{eqdfq} are equivalent to the
vanishing of $\s^*\I$.   Suppose that $\s^* \I = 0$.  By
Corollary~\ref{chicorollary},
$d \chi + \chi \wedge \chi = 0$.  It follows that there exists an
analytic solution $g \in \GL_n(R)$ on $\De$ to the integrable
differential equation $d g = g.\chi$.  In particular, $\chi = g^{-1} dg$.

We observe that $\s^*
\I = \{0\}$ if and only if $\s^*(\Th_i -\Ad(g_i) (\chi)) = 0$
and $\s^* \Xi_i - [\chi, \a]= 0$ for all $i$, and the vanishing of these forms is
equivalent to conditions \eqref{eqflow} and \eqref{eqcurve} of
Corollary~\ref{eqdfq}.  Also, by Lemma~\ref{modI2}, $\s^* \I =
\{0\}$ implies $\s^* (d\Up + \Up \wedge\Up) = 0$, which is equivalent
to condition \eqref{eqdO}.

\end{proof}

Theorem~\ref{int} and Lemma~\ref{modI2} immediately show that the
third conditions in Theorem~\ref{defeq} and Corollary~\ref{eqdfq} are
redundant.

\begin{theorem}\label{redundant}  Let  $( \bfg, \bV, \tn)$ be a good framed
  deformation.  The following statements are equivalent.
  \begin{enumerate} \item $( \bfg, \bV, \tn)$ is integrable.
\item  $( \bfg, \bV, \tn)$  is $\GL_n(R)$-gauge-equivalent to a
  deformation satisfying the first two conditions of
  Theorem~\ref{defeq}.
\item There exists $g\in\GL_n(R)$ such that the first two conditions
  of Corollary~\ref{eqdfq} are satisfied.
\end{enumerate}
\end{theorem}

\begin{proof}[Proof of Theorem~\ref{int}, part~\eqref{intb}]
Since $q^* \bI = \I$, and $\I$ is a Pfaffian system by Theorem~\ref{int}, one immediately
sees that $\bI$ is Pfaffian.  Specifically, the rank of $\bI^1$ at any point  $m \in \tM$
is equal to the rank of $\I^1$ at a point in $q^{-1} (m)$.  

In order to prove integrability, suppose that $\om \in \bI^1$.  Then,
$d (q^* \om) = q^* (d \om) \in \I$.  It now suffices to work on a
neighborhood $W \subset \tM$.  Choosing a local section $\s :W \to
q^{-1} (W)$ of the principal $\GL_n(\C)$-bundle, we see that $\s^*
\I|_{q^{-1} (W)} = \bI|_{W}$.  We deduce that $d \om = \s^* q^* (d
\om) \in \bI|_{W}$.

Finally, suppose that $f : \De \to \tM$.  Again, we choose a local
section $\s : W \to q^{-1} (W)$ of a neighborhood $W$ containing the
image of $\De$.  By the first part of Theorem~\ref{int}, $f$
corresponds to a framed integrable deformation if and only if 
$f^* \s^* (\I) = 0$.  However, this is equivalent to the statement
that $f^* (\bI) = 0$ since $\s^* (\I) = \bI |_{W}$.

\end{proof}

\section{Example}\label{example}
In this section, we give an explicit example of the system of
equations constructed above.  We will consider a space of rank $n$
meromorphic connections on $\proj$ with $m$ singularities of slope
$\frac{1}{n}$.  Let $\bfx$ be a set of $m$ finite points, and set $P_i
= I_i$ and $r_i = 1$ for all $i$.  If $z$ is the usual coordinate on
$\proj$, we write $z_i = (z - \xi_i).$ Accordingly,
$\o_i=\o_{T_i}=N+z_i E \in \fP^1_i$, where $N$ is the regular
nilpotent matrix in Jordan form and $E$ is the elementary matrix
$E_{n1}$.  Note that $\o_i^{-1} = N' + \frac{1}{z_i}E'$, where $N'$
and $E'$ are the transposes of $N$ and $E$ respectively.

We choose our one form to be $\nu =dz$,  so $\tau_\nu = \pd_z = \frac{d}{dz}$.
Choose a point $(U_i g_i , \a_i)_{i \in I} \in \tMg$ corresponding to
a connection $\n$, and write $[\n] = \a \nu$.  Thus, $\a \in \a_{i
  \nu} + \fI^1_i$ at each $i$ and $\sum_{i \in I } \Res_i (\a \nu) =
0$.  To simplify calculations, we assume that the normalized formal
type of $\n$ at $\xi_i$ has the representative
$\frac{1}{z_i}(-\frac{a_i}{n} \o_i^{-1} + H_T)$ under the pairing
$\langle, \rangle_\nu$.

We write $A_{\tM, i} = \o_i^{-1} da_i$ and $\Ad(g_i)(\a) =
\frac{1}{z_i}(-\frac{a_i}{n} \o_i^{-1} - \frac{1}{n} (D_i + X_i)+ H_T)
+ \fg_i$ for some traceless diagonal matrix $D_i$ and $X_i \in \fu_i$.
We let $\vr_i$ denote the residue term of $\a\nu$ at $x_i$.
Explicitly, $\vr_i=\phi_i (z_i \a
)=\Ad(g^{-1})(-\frac{1}{n}(N'+D_i+X_i)+H_T)$.

\begin{proposition} The isomonodromy equations for the system above are
\begin{subequations}\label{exmain}
\begin{gather}\label{ex}
\begin{multlined}
( (d g_i) g_i^{-1}, g_i) =
\left(\left[\sum_{j \in I \backslash i} \frac{1}{\xi_i - \xi_j}  \Ad(g_ig_j^{-1}) (E' da_j)\right]
 - \left[ N' da_i + D_i \frac{da_i}{a_i} \right] + 
 \fu_i, g_i \right) \\ \in   (\fu_i\backslash \gl_n(\C)) \times_{U_i}
\GL_n(\C)\text{ and }
\end{multlined}
\\
\label{exsecond} \begin{multlined} d \vr_i = \sum_{j \in I \backslash
    \{i\}} \left( \frac{1}{\xi_j - \xi_i} \left ([\Ad(g_j^{-1}) E',
      \vr_i] da_j+ [\Ad(g_i^{-1}) E', \vr_j]da_i \right) \right.
  \\
  \left.+\frac{1}{n} \frac{1}{(\xi_j - \xi_i)^2} [\Ad(g_i^{-1})E',
    \Ad(g_j^{-1})E'] d(a_i a_j) \right).
\end{multlined}
\end{gather}
\end{subequations}
\end{proposition}

\begin{rmk} In this system of partial differential equations, the
  independent variables are the $a_i$'s, which are essentially the
  formal types, while the dependent variables are the framings $g_i$
  and the residues $\vr_i$ of $\a\nu$ (which can be written explicitly
  in terms of $g_i$, $D_i$, and $X_i$).
\end{rmk}

\begin{proof}

First, observe that by definition,
\begin{equation*}
\Phi_i =  (\o_i^{-1}  da_i + D_i \frac{da_i}{a_i}) + \Om^1_{\De} (\fI_i^1).
\end{equation*}
Therefore, 
\begin{equation*}
\Up^0 = \sum_{i \in I} \Ad(g_i^{-1})\left(z_i^{-1} E'da_i \right),
\end{equation*}
and 
\begin{equation*}
\phi_i (\Up^0) = \sum_{j \in I \backslash \{i\}} \frac{1}{\xi_i - \xi_j}  \Ad(g_j^{-1}) (E' da_j).
\end{equation*}
We conclude that equation~\eqref{flow} from Theorem~\ref{defeq} is equivalent to \eqref{ex}.


Now, we consider equation~\eqref{curv} of Theorem~\ref{defeq}.  At $\xi_i$,
applying $\Ad(g_i)$ to the principal part of the curvature gives us
\begin{multline}\label{ex2}
-z_i^{-2} E' da_i \in \\
-\frac{1}{nz_i } (\o_i^{-1}da_i + d D_i + d X_i)  +
[\Ad(g_i)(\Up^0)- (dg_i) g_i^{-1}, \Ad(g_i)(\a)] + \Om^1_\De(\fg_i).
\end{multline}
First, we calculate $[z_i^{-1} E', \Ad(g_i)(\a)] + \fg_i$.  
Observe that
$[z_i^{-1} E', \Ad(g_i) (\a) dz]$ is a one form on $\proj_\De$ with poles along $\bfx$.
Thus, 
\begin{equation}\label{ex3}
[z_i^{-1} E', \Ad (g_i) (\a) dz] \in (-\frac{1}{n z_i} [z_i^{-1}E' , N' a_i + D_i + X_i-nH_T] + \frac{1}{z_i} Y
)dz+ \fg_i dz,
\end{equation}
where $Y$ is the residue term.  By the residue theorem,
$Y + \sum_{j \in I} \Res_j ([z_i^{-1} E', \Ad(g_i) (\a) dz]) = 0$.
We conclude that
\begin{equation*}
\begin{aligned}
Y & = \frac{1}{n} 
\sum_{j \in I \backslash \{i\}} \Res_j ( [z_i^{-1}E', \Ad(g_i g_j^{-1}) (\o_i^{-1}a_j + D_j + X_j
- n H_T)] \frac{dz}{z_j})\\
& = 
\frac{1}{n} 
\sum_{j \in I \backslash \{i\}} \frac{1}{\xi_j - \xi_i} [E', \Ad(g_i g_j^{-1}) (N'a_j + D_j + X_j - n H_T)]
- \frac{1}{ (\xi_j - \xi_i)^2} [ E', \Ad(g_i g_j^{-1}) E' a_j].
\end{aligned}
\end{equation*}

By Proposition~\ref{totalconnection} we see that
\begin{equation}\label{ex1}
  - z_i^{-2}E' da_i 
  =  -\frac{1}{n z_i} (\o_i^{-1} da_i) + \frac{1}{z_i}[\o_i^{-1} da_i, 
  H_T] .
\end{equation}

Comparing \eqref{ex2}, \eqref{ex3}, and \eqref{ex1}, we obtain the
condition
\begin{multline}\label{ex4}
-\frac{1}{nz_i} \left(  d (D_i + X_i) 
+ [z_i^{-1}E'  da_i, N' a_i + D_i + X_i - n H_T] \right)+
\\ \frac{1}{z_i} Y da_i
+ [\Ad(g_i)(\Up^0) - z_i^{-1}  E' da_i- (dg_i) g_i^{-1}, \Ad(g_i) (\a)] - \frac{1}{z_i} [\o_i^{-1} da_i, H_T]
\in \Om^1_{\De} (\fg_i).
\end{multline}

Now, by \eqref{ex},
\begin{multline}\label{ex5}
[\Ad(g_i)\Up^0 - z_i^{-1} E' da_i - (dg_i) g_i^{-1}, z_i^{-1} E' a_i]  
\in \\
[\phi_i (\Ad(g_i)\Up^0) - (dg_i) g_i^{-1}, z_i^{-1} E' a_i] -
\sum_{j \in I \backslash \{i\}} \frac{1}{(\xi_j - \xi_i)^2}[\Ad(g_i g_j^{-1} )(E'), E' ] a_i da_j +
\Om^1_{\De}(\fg_i^1)  \\
= [ N' da_i + D_i \frac{da_i}{a_i}, z_i^{-1} E'a_i] -
\sum_{j \in I \backslash \{i\}} \frac{1}{(\xi_j - \xi_i)^2}[\Ad(g_i g_j^{-1} )(E'), E' ] a_i da_j + 
\Om^1_{\De}(\fg^1_i).
\end{multline}
On the other hand,
\begin{multline}\label{ex6}
[\Ad(g_i)\Up^0 - z_i^{-1} E'da_i, N'a_i + D_i + X_i-n H_T]  \in \\
 [\phi_i (\Ad(g_i)\Up^0), N'a_i + D_i + X_i- n H_T] + \Om^1_{\De} (\fg_i^1) \\
=[\sum_{j \in I \backslash \{i\}}\frac{1}{\xi_i - \xi_j} \Ad(g_i g_j^{-1})(E' da_j), 
N'a_i + D_i + X_i-n H_T] + \Om^1_{\De}(\fg_i^1).
\end{multline}
Finally,
\begin{equation}\label{ex7}
[ z_i^{-1} E' da_i, H_T] - [\o_i^{-1} da_i, H_T] = [-N' da_i, H_T].
\end{equation}
We substitute \eqref{ex5}, \eqref{ex6}, and \eqref{ex7} into
\eqref{ex4}:
\begin{multline*}
d ( D_i + X_i) - [(dg_i) g_i^{-1}, N' a_i+ D_i+  X_i- n H_T]  \equiv 
[-N' da_i, n H_T]+
\\
\sum_{j \in I \backslash \{i\}} \frac{1}{\xi_j - \xi_i} [E' da_i, \Ad(g_i g_j^{-1}) (N'a_j + D_j + X_j - n H_T)]
-\frac{1}{ (\xi_j - \xi_i)^2} [ E', \Ad(g_i g_j^{-1}) E' ]a_j da_i \\
 + \sum_{j \in I \backslash \{i\}} \frac{1}{(\xi_j - \xi_i)^2}[\Ad(g_i g_j^{-1} )(E'), E' ] a_i da_j \\ -
 \sum_{j \in I \backslash \{i\}}\frac{1}{\xi_i - \xi_j} [\Ad(g_i g_j^{-1})(E' da_j), 
N'a_i + D_i + X_i-n H_T]  \pmod{ \Om^1_{\De}(\fg_i^1)}.
\end{multline*}

Since $[N'da_i,  H_T] = \frac{1}{n} N' da_i$, applying
$\Ad(g_i^{-1})$ and dividing both sides by $-n $ shows that this equation
is equivalent to \eqref{exsecond}.
\end{proof}

\end{document}